\documentclass[10pt]{article}
\usepackage[margin=.5in]{geometry}
\usepackage{fullpage}
\usepackage{multirow}
\usepackage{amsgen,amsmath,amstext,amsbsy,amsopn,amssymb, amsthm}
\usepackage{bm}
\usepackage{url}
\usepackage{graphicx}
\usepackage{caption}
\usepackage{subcaption}
\usepackage{algorithm}
\usepackage{tablefootnote}
\usepackage{threeparttable}
\usepackage{dsfont}

\usepackage[mathscr]{eucal}

\usepackage{amsbsy}

\usepackage{color}
\usepackage{comment}
\usepackage{enumitem}
\usepackage{booktabs}
\usepackage{algpseudocode}
\usepackage{longtable}
\usepackage{mleftright}

\usepackage{thmtools}

\usepackage{fixme}
\fxsetup{
  status=draft,
  theme=color,
  silent,
}
\fxsetup{marginface=\linespread{1}\tiny}
\FXRegisterAuthor{tk}{tke}{TK} 
\FXRegisterAuthor{rt}{rte}{RT}
\FXRegisterAuthor{az}{aze}{AZ}

\usepackage{hyperref} % Should be last
\hypersetup{
	colorlinks=true,
	linkcolor=blue,
	filecolor=blue,      
	urlcolor=red,
	citecolor=blue,
}
\usepackage[capitalize,nameinlink]{cleveref} % except for this one

\DeclareMathAlphabet\mathbfcal{OMS}{cmsy}{b}{n}

\newcommand{\cX}{\mathcal X}

\newcommand{\rank}{\rm rank}
\newtheorem{proposition}{Proposition}
\setlist[description]{style=unboxed,leftmargin=.5em}
\setlist[itemize]{style=sameline,leftmargin=2em}

\ifpdf
  \DeclareGraphicsExtensions{.eps,.pdf,.png,.jpg}
\else
  \DeclareGraphicsExtensions{.eps}
\fi

\newtheorem{theorem}{Theorem}[section]
\newtheorem{lemma}{Lemma}[section]
\newtheorem{remark}{Remark}[section]
\newtheorem{assumption}{Assumption}[section]

\allowdisplaybreaks

\title{Group Sparse-based Tensor CP
Decomposition: Model, Algorithms, and Applications in Chemometrics}

\author{Zihao Wang\footnote{School of Mathematics, Hunan University, Changsha, Hunan 410082,  P. R. China; Email: \texttt{wzh97@hnu.edu.cn}.}, ~ Minru Bai\footnote{Corresponding author. School of Mathematics, Hunan University, Changsha, Hunan 410082,  P. R. China; Email: \texttt{minru-bai@163.com}.}, ~ Liang Chen\footnote{School of Mathematics, Hunan University, Changsha, Hunan 410082,  P. R. China; Email: \texttt{chl@hnu.edu.cn}.}, ~ and ~ Xueying Zhao\footnote{College of Sciences, Tianjin University of Science and Technology, Tianjin 300457, P.R.China; Email: \texttt{xueying\_zhao@tust.edu.cn}. \\ \indent The research of this work was supported by the National Key R\&D Program of China (No. 2021YFA1001300), 
the National Natural Science Foundation of China (Nos. 12271150, 12071399, 11971159), 
the Natural Science Foundation of Hunan Province (No. 2023JJ10001), 
the Science and Technology Innovation Program of Hunan Province (No. 2022RC1190).}}

\date{\today}

\begin{document}
	
\maketitle

\begin{abstract}
The CANDECOMP/PARAFAC (or Canonical polyadic, CP) decomposition of tensors has numerous applications in various fields, such as chemometrics, signal processing, machine learning, etc. Tensor CP decomposition assumes the knowledge of the exact CP rank, i.e., the total number of rank-one components of a tensor. However, accurately estimating the CP rank is very challenging.  In this work, to address this issue, we prove that the CP rank can be exactly estimated by minimizing the group sparsity of any one of the factor matrices under the unit length constraints on the columns of the other factor matrices. Based on this result, we propose a CP decomposition model with group sparse regularization, which integrates the rank estimation and the tensor decomposition as an optimization problem, whose set of optimal solutions is proved to be nonempty. To solve the proposed model, we propose a double-loop block-coordinate proximal gradient descent algorithm with extrapolation and prove that each accumulation point of the sequence generated by the algorithm is a stationary point of the proposed model. Furthermore,  we incorporate a rank reduction strategy into the algorithm to reduce the computational complexity. Finally, we apply the proposed model and algorithms to the component separation problem in chemometrics using real data. Numerical experiments demonstrate the robustness and effectiveness of the proposed methods.
\end{abstract}


\section{Introduction}\label{sec1}
\indent
\par

Tensor decomposition can reveal underlying information in data analysis and is widely used in various fields \cite{Kolda}. 
The CANDECOMP/PARAFAC (CP) decomposition \cite{Carroll, PARAFAC}, also known as the canonical polyadic decomposition \cite{hitchcock1927tensor}, is one of the most important tasks in the decomposition of tensors. 
The CP decomposition of a given tensor $\mathcal{X}\in\mathds{R}^{n_1\times\cdots\times n_N}$ factorizes $\mathcal{X}$ into a sum of rank-one component tensors and takes the following form:
\begin{equation}\label{CPN}
    \mathcal{X}=[\![{\mathbf{A}}^{(1)},\cdots,{\mathbf{A}}^{(N)}]\!]
    :=\sum_{r=1}^R \mathbf{a}^{(1)}_r \circ \cdots \circ {\mathbf{a}}^{(N)}_{r},
\end{equation}
where $R$ is a positive integer,  $\mathbf{a}^{(i)}_r\in\mathds{R}^{n_i}, r=1,\cdots, R$, represents the $r$-th column of the factor matrix $\mathbf{A}^{(i)}\in\mathds{R}^{n_i\times R}, i=1,\cdots,N$, and the  symbol $\circ$ represents the outer product. 
In this expression, the CP rank of $\mathcal{X}$, denoted $\rank_{cp}(\mathcal{X})$, is defined as the smallest number of rank-one tensors, that is, the smallest $R$ such that \eqref{CPN} holds \cite{Kolda}.
If $R = \rank_{cp}(\mathcal{X})$ in \eqref{CPN}, then it is called the $\rank$ CP decomposition \cite{Kolda}.
The uniqueness of the decomposition of CP under mild assumptions \cite{domanov2013uniqueness, kruskal1977three, sidiropoulos2000uniqueness} leads to countless applications in
chemometrics \cite{Lenhardt2015, wu2024new, wu2024new1}, signal processing \cite{Nion2010, Sidiropoulos2000}, machine learning \cite{Nakatsuji2017, Signoretto2013}, neuroscience \cite{Cong2015, Davidson2013}, and hyperspectral imaging \cite{Sun2020, Xu2021}, to name only a few. One may refer to surveys such as \cite{Acar2009, Cichocki2016, Cichocki2017, Kolda} for a more general background.

Numerous CP decomposition models and the corresponding algorithms have been developed, with the assumption of a fixed CP rank, i.e., the number of rank-one components. The most widely studied CP decomposition model is the unconstrained optimization problem by minimizing the least squares (LS) loss function.
Several algorithms have been proposed to solve this unconstrained optimization model.
Harshman \cite{PARAFAC} first introduced the parallel factor analysis (PARAFAC) algorithm. Wu et al. \cite{ATLD} proposed an alternating trilinear decomposition algorithm (ATLD) for the trilinear decomposition model in chemometrics, i.e., the CP decomposition model. Kolda et al. \cite{Battaglino2018, Larsen2022} developed randomized CP decomposition algorithms for large-scale dense and sparse tensors. Wang et al. \cite{Wang2023, Wang2023a}  proposed accelerated stochastic gradient descent algorithms for large-scale tensor CP decomposition. 
In recent years, CP decomposition models with constraints, such as orthogonal, nonnegative and column unit constraints, have been proposed for practical applications.  
For example, Yang \cite{Yang2020} introduced an $\epsilon$ alternating least squares algorithm to solve the CP decomposition model with orthogonal constraints.  
Wang and Cong \cite{Wang2021} studied the CP decomposition model with nonnegative constraints and introduced an inexact accelerated proximal gradient (iAPG) algorithm to solve it.
Wang and Bai  \cite{WangCOAP} introduced a CP decomposition model with nonnegative and column unit constraints for the component separation problem in complex chemical systems, and proposed an accelerated inexact block coordinate descent algorithm to solve the model.

Currently, one crucial challenge in CP decomposition models and algorithms is the requirement for accurately estimating the CP rank in advance. 
Deviations from the true CP rank may increase fitting errors and cause factor degeneracy.
However, the tensor CP rank is NP-hard to compute \cite{Kolda} and there is no finite algorithm for determining the rank of a tensor \cite{Haastad1990, Kruskal1989RankDA}.
These observations naturally lead us to raise the following question: Can we establish a tensor CP decomposition model with CP rank estimation simultaneously?

To address this question, in this paper, 
 we introduce a CP decomposition model with group sparse regularization given by
\begin{equation}
	\begin{array}{lll}\label{CPGS}
		\min\limits_{\mathbf{A}^{(1)},\cdots, \mathbf{A}^{(N)}} \quad & \frac{1}{2}\|\mathcal{X}- [\![\mathbf{A}^{(1)},\cdots, \mathbf{A}^{(N)}]\!]\|^2+\lambda\|\mathbf{A}^{(N)}\|_{2,0} \\
		\text { s.t. } \quad& \|\mathbf{a}^{(i)}_r\|=1, i=1,\cdots,N-1, r=1\cdots R,
	\end{array}
\end{equation}
 where $\mathcal{X}\in\mathds{R}^{n_1\times\cdots\times n_N}$ is a given tensor and $\|\mathbf{A}^{(N)}\|_{2,0}:=\sum_{r=1}^R(\|\mathbf{a}^{(N)}_r\|)^0$. 
We establish the equivalence between the CP rank of a given tensor and minimizing the $\ell_{2,0}$ norm on one of the factor matrices under the unit length constraints on the columns of the others.

To solve the problem \eqref{CPGS}, 
we develop a double-loop block-coordinate proximal gradient descent algorithm with extrapolation. 
In this algorithm, each factor matrix is treated as a block and solved approximately by a block-coordinate proximal gradient descent algorithm with finitely many cycles. 
Additionally, extrapolation is applied after each factor matrix is updated, and a check to ensure the reduction of objective function value is performed after each outer iteration.
We prove that each limit point of the sequence generated by the proposed algorithm is a stationary point of the model \eqref{CPGS}.
Furthermore, we prove that for a convergent sequence $\{(\mathbf{A}^{(1)^k},\cdots,\mathbf{A}^{(N)^k})\}$, the nonzero columns of $\mathbf{A}^{(N)^k}$ remain unchanged after a certain number of iterations. 
Based on this result, we develop a rank reduction strategy to reduce computational complexity.
The rank reduction strategy works as follows: if the nonzero columns of $\mathbf{A}^{(N)^{k}}$ remain unchanged, removing the zero columns in $\mathbf{A}^{(N)^{k}}$ and the corresponding columns in $\mathbf{A}^{(i)^{k}}, i=1,\cdots, N-1$, and allowing the iterations to continue. 

Eventually, we apply the proposed model and algorithms to the component separation problem in chemometrics.
This problem involves resolving chemically meaningful independent component information from mixed-response signals, which can be formulated as a CP decomposition model.
Accurate estimation of the number of components is crucial in this problem since incorrect estimation can lead to separation failure. 
Numerical experiments conducted on third-order real chemical data demonstrate that our proposed methods can precisely separate the components and 
accurately estimate their concentrations, even when the initial number of components is overestimated.
This suggests that our proposed methods can simultaneously estimate the CP rank and obtain the CP decomposition of a given tensor.

The main contributions of this paper are threefold.
\begin{itemize}
\item We prove that determining the CP rank of a given tensor is equivalent to minimizing the $\ell_{2,0}$ norm on one of the factor matrices under the unit length constraints on the columns of the others. Based on this equivalence, we propose a CP decomposition model with group sparse regularization and unit length constraints, which can simultaneously estimate the CP rank and obtain the CP decomposition. 
\item We develop a double-loop block-coordinate proximal gradient descent algorithm with extrapolation to solve the proposed model and prove that each limit point of the sequence generated by this algorithm is a stationary point of the model.
Furthermore, we demonstrate that for a convergent sequence $\{(\mathbf{A}^{(1)^k},\cdots,\mathbf{A}^{(N)^k})\}$, the nonzero columns of $\mathbf{A}^{(N)^k}$ remain unchanged after a certain number of iterations. Based on this observation, we develop a rank reduction strategy to reduce the computational complexity.
\item We apply the proposed models and algorithms to solve the component separation problem in chemometrics.
Numerical experiments on real data demonstrate the robustness and effectiveness of our proposed methods. 
\end{itemize}

The remainder of this paper is organized as follows. 
First, in Section \ref{Preliminaries},  some notations and definitions are introduced. 
Next, in Section \ref{sec3}, we prove the equivalent surrogate of the CP rank of a tensor and present a CP decomposition model with group sparse regularization and unit length constraints.  
After that, in Section \ref{sec4}, we develop a double-loop block-coordinate proximal gradient descent algorithm with extrapolation to solve our proposed model. We also provide a convergence analysis and incorporate a rank reduction strategy into the algorithm.
Subsequently, in Section \ref{sec5}, we apply the proposed methods to the component separation problem in chemometrics. 
Eventually, the conclusions are drawn in Section \ref{sec6}.

%%%%%%%%%%%
\section{{Preliminaries}}	 
\label{Preliminaries}
\indent
\par
We introduce the definitions and notation used throughout this paper. 
Scalars, vectors, matrices, and tensors are denoted as lowercase letters ($x$), lowercase boldface letters ($\mathbf{x}$), boldface capital letters ($\mathbf{X}$), and calligraphic letters ($\mathcal{X}$), respectively.
Sets are denoted as hollow letters, for example, $ \mathbb{X}$.
The real field is denoted by $\mathds{R}$. 
For a $N$-th order tensor $ \mathcal{X} \in \mathds{R}^{n_1 \times\cdots\times n_N}$, we denote its $(i_1,\cdots, i_N)$-th entry as ${X}_{i_1,\cdots,i_N}$.	
The $i$-th column of the matrix $ \mathbf{X} $ is denoted as $\mathbf{x}_i$.
The transpose of the matrix $\mathbf{X}$ is denoted by $\mathbf{X}^\top$. 
The Frobenius norms for a tensor $\mathcal{X}$ and a matrix $\mathbf{X}$ are denoted by $\|\mathcal{X}\|$ and $\|\mathbf{X}\|$, respectively. 
The Euclidean norm for a vector $\mathbf{x}$ is denoted by $\|\mathbf{x}\|$. 
We use the symbols $\circ $, $\otimes$, and $\odot$ to represent the outer, Kronecker, and Khatri-Rao products, respectively. 
For a real vector $\mathbf{x}$, its zero norm $\|\mathbf{x}\|_0$ is the cardinality of its 
nonzero components. 
For a given matrix $\mathbf{X}=(\mathbf{X}_1, \cdots, \mathbf{X}_m)$, 
its $\ell_{2,0}$ norm is defined by
$\|\mathbf{X}\|_{2,0}:=\|(\|\mathbf{X}_1\|, \cdots, \|\mathbf{X}_m\|)^\top\|_0$, and
we denote 
\begin{equation*}
    \mathbf{X}_{<i}:=(\mathbf{X}_1,\cdots,\mathbf{X}_{i-1}),\quad \mathbf{X}_{\geq i}:=(\mathbf{X}_{i},\cdots,\mathbf{X}_{m}).
\end{equation*}
We use ${\rm supp}(\mathbf{X})$ to represent the index set of the nonzero columns of $\mathbf{X}$. 
% For a proper and lower semicontinuous function $f : \mathds{R}^n \to (-\infty,\infty]$ and a real parameter $\lambda >0 $, the proximal mapping \cite[Definition 1.22]{rockafellar2009variational} of $f$ is defined by 
% $${\rm prox}_{\lambda}^{f}(\mathbf{x}):=\argmin\limits_{\mathbf{w}} 
% \big\{f(\mathbf{w})+\frac{1}{2\lambda}\|\mathbf{w}-\mathbf{x}\|^2 \big\}. 
% $$ 

\section{A regularized CP decomposition model}
\label{sec3}
\indent
\par
This section proposes a regularized CP decomposition model to obtain a rank CP decomposition of a given tensor. 
Recall that for a given tensor $\mathcal{X}\in \mathds{R}^{n_1\times \cdots \times n_N }$, as shown in \eqref{CPN}, the CP decomposition factorizes $\mathcal{X}$ into a sum of rank-one tensors. 
The CP rank of $\mathcal{X}$, denoted by $\rank_{cp}(\mathcal{X})$, 
is defined as the smallest number of rank-one tensors that can constitute a CP decomposition, i.e., the smallest $R$ in \eqref{CPN}.
As mentioned in Section \ref{sec1}, most current CP decomposition models require an accurate prior estimation of the CP rank of $\mathcal{X}$ to obtain a satisfactory decomposition, and it is very challenging to directly compute the CP rank of $\mathcal{X}$. 

To address this issue, we first establish the equivalence between the CP rank and the optimal value of the optimization problem, which minimizes the group sparsity on any one of the factor matrices, while under the unit length constraints on the columns of the others.

\begin{theorem}
For a integer ${R} \ge \rank_{cp}(\mathcal{X})$, we have
\begin{equation}\nonumber
{\rm rank_{cp}}(\mathcal{X})=\min\left\{\|\mathbf{A}^{(N)} \|_{2,0}
\,\Big\vert\,   
\begin{array}{ll}
\mathcal{X}=[\![{\mathbf{A}}^{(1)},\cdots,{\mathbf{A}}^{(N)}]\!], 
\\
\|\mathbf{a}^{(i)}_r\|=1,\, r=1,\cdots, {R},\, i=1,\cdots, N-1
\end{array}\right\},
\end{equation}
where $[\![{\mathbf{A}}^{(1)},\cdots,{\mathbf{A}}^{(N)}]\!]=\sum\limits_{r=1}^R  \mathbf{a}^{(1)}_r \circ \cdots \circ \mathbf{a}^{(N)}_r$.
% }
\end{theorem}
\begin{proof}
% Let $\mathbf{a}^{(i)}_r: =\overline{\mathbf{a}}^{(i)}_{r}, i=1,\cdots,N-1$ and $\mathbf{a}^{(N)}_r: =\lambda_r\overline{\mathbf{a}}^{(N)}_{r}$.
By the definition of the CP rank of a tensor, we have 
$$
\begin{array}{lll}
&{\rm rank_{cp}}(\mathcal{X})\\[1mm]
&=\min\big\{R\vert
\mathcal{X}=\sum\limits_{r=1}^R \lambda_r {\mathbf{a}}^{(1)}_r \circ \cdots \circ \overline{\mathbf{a}}^{(N)}_{r},\|{\mathbf{a}}^{(i)}_r\|=\|\overline{\mathbf{a}}^{(N)}_r\|=1  ,i=1,\cdots,N-1
\big\}
\\[1mm]
&=\min  \big\{\|\lambda\|_0\vert \mathcal{X}=\sum\limits_{r=1}^{{R}} \lambda_r {\mathbf{a}}^{(1)}_r \circ \cdots \circ \overline{\mathbf{a}}^{(N)}_{r},\|{\mathbf{a}}^{(i)}_r\|=\|\overline{\mathbf{a}}^{(N)}_r\|=1  ,i=1,\cdots,N-1 \big\},
\end{array}
$$
where $\lambda=(\lambda_1,\cdots,\lambda_R)^\top\in \mathds{R}^{R}$. Denote ${\mathbf{a}}^{(N)}_{r}:=\lambda_r\overline{\mathbf{a}}^{(N)}_{r}, r=1,\cdots,R$, and $\mathbf{A}^{(N)}:=(\mathbf{a}_1^{(N)},\cdots,\mathbf{a}_R^{(N)})$, then we have
	\begin{equation}
		\nonumber
			{\rm rank_{cp}}(\mathcal{X})=\min  \big\{ \|\mathbf{A}^{(N)} \|_{2,0}\vert \mathcal{X}=\sum\limits_{r=1}^{{R}} \mathbf{a}^{(1)}_r \circ \cdots \circ {\mathbf{a}}^{(N)}_{r},\|\mathbf{a}^{(i)}_r\|=1,i=1,\cdots,N-1  \big\},
	\end{equation}
	where it comes from  that $\|\mathbf{a}^{(N)}_r\|=\|\lambda_r\overline{\mathbf{a}}^{(N)}_r\|=0$ and $\|\overline{\mathbf{a}}^{(N)}_r\|=1$ if and only if $\lambda_r=0$, and $$\sum_{r=1}^R \lambda_r {\mathbf{a}}^{(1)}_r \circ\cdots \circ \overline{\mathbf{a}}^{(N)}_{r}=\sum_{r=1}^R  \mathbf{a}^{(1)}_r \circ \cdots \circ(\lambda_r\overline{\mathbf{a}}^{(N)}_{r})=\sum_{r=1}^R \mathbf{a}^{(1)}_r \circ \cdots \circ {\mathbf{a}}^{(N)}_{r}.$$ This completes the proof.
\end{proof}

Based on the above theorem,  we propose the following CP decomposition model with group sparse regularization and unit length constraints (CPD\_GSU),
\begin{equation}
	\begin{array}{lll}\label{model}
		\min\limits_{{\mathbf{A}}^{(1)},\cdots,{\mathbf{A}}^{(N)}} \quad & \frac{1}{2}\|\mathcal{X}-[\![{\mathbf{A}}^{(1)},\cdots,{\mathbf{A}}^{(N)}]\!]\|^2+\lambda\|\mathbf{A}^{(N)}\|_{2,0} \\
		\text { s.t. } \quad &\|\mathbf{a}^{(i)}_r\|=1, r=1,\cdots, R, i=1,\cdots, N-1,
	\end{array}
\end{equation}
where $ \mathcal{X}\in \mathds{R}^{{n_1} \times \cdots \times {n_N}} $ is a given tensor, 
$[\![{\mathbf{A}}^{(1)},\cdots,{\mathbf{A}}^{(N)}]\!]=\sum_{r=1}^R \mathbf{a}^{(1)}_r \circ \cdots \circ {\mathbf{a}}^{(N)}_{r}$ and $\|\mathbf{A}^{(N)}\|_{2,0}=\sum_{r=1}^R(\|\mathbf{a}^{(N)}_r\|)^0$. $R$ is a positive integer, and $\mathbf{A}^{(i)}=(\mathbf{a}_1^{(i)},\cdots\mathbf{a}_R^{(i)})\in\mathds{R}^{n_i\times R},i=1,\cdots N$.
% The $\ell_{2,0}$ norm in this model characterizes the CP rank of the tensor. 

By letting $f(\mathbf{A}^{(1)}, \cdots, \mathbf{A}^{(N)}):=\frac{1}{2}\|\mathcal{X}-[\![{\mathbf{A}}^{(1)},\cdots,{\mathbf{A}}^{(N)}]\!]\|^2$, 
we can rewrite the model \eqref{model} into an unconstrained optimization problem as follows:
\begin{equation}\label{modelF}
	\min_{\mathbf{A}^{(1)}, \cdots, \mathbf{A}^{(N)} }F(\mathbf{A}^{(1)}, \cdots, \mathbf{A}^{(N)})=f(\mathbf{A}^{(1)}, \cdots, \mathbf{A}^{(N)}) +\lambda \|  \mathbf{A}^{(N)}\|_{2,0}+\sum\limits_{i=1}^{N-1}\delta_{\mathbb{A}^{(i)}}( \mathbf{A}^{(i)}),
\end{equation}
where  $ \delta_{\mathbb{A}^{(i)}}(x) $ is the indicator function and $\mathbb {A}^{(i)}:= \{\mathbf{A}^{(i)}\in \mathds{R}^{n_i\times R}|\|\mathbf{a}_r^{(i)}\|=1,r=1,\cdots,R\}$.

We show that the objective function $F$ in \eqref{modelF} is coercive and the set of optimal solutions of  \eqref{model} is nonempty in the following proposition.
\begin{proposition}
	The objective function $F$ in \eqref{modelF}  is a coercive function, and the set of optimal solutions of \eqref{model} is nonempty.
\end{proposition}
\begin{proof}
	By the definition of $ [\![\mathbf{A}^{(1)}, \cdots, \mathbf{A}^{(N)}]\!]$, we have
   \begin{equation}\nonumber
       \begin{array}{ll}
		F(\mathbf{A}^{(1)}, \cdots, \mathbf{A}^{(N)})&=\frac{1}{2}\|\mathcal{X}- [\![\mathbf{A}^{(1)}, \cdots, \mathbf{A}^{(N)}]\!] \|^2 +\lambda \| \mathbf{A}^{(N)}\|_{2,0}+\sum\limits_{i=1}^{N-1}\delta_{\mathbb{A}^{(i)}}( \mathbf{A}^{(i)})\\[1mm] 
        \qquad
		&=\frac{1}{2}{\sum\limits_{j_1=1}^{{n_1}}}\cdots{\sum\limits_{j_N=1}^{{n_N}}}({{X}}_{j_1\ldots j_N}-{\sum\limits_{r=1}^{R}}{a}^{(1)}_{{j_1 r}}\cdots {a}^{(N)}_{{j_N r}})^2\\
        &\quad +\lambda \| \mathbf{A}^{(N)}\|_{2,0}+\sum\limits_{i=1}^{N-1}\delta_{\mathbb{A}^{(i)}}( \mathbf{A}^{(i)}).
	\end{array}
   \end{equation}
   
   We first show the objective function $F$ in \eqref{modelF}  is a coercive function.
   If there exists $\mathbf{A}^{(i)}\notin \mathbb{A}$, $i\in\{1,\cdots,N-1\}$,  we have $F(\mathbf{A}^{(1)}, \cdots, \mathbf{A}^{(N)}) = +\infty$ by  the definition of indicator functions. 
   If  $\mathbf{A}^{(i)}\in \mathbb{A}$ holds for all $i=1,\cdots,N-1$,  we have $\|\mathbf{a}^{(i)}_r\|=1,i=1,\cdots,N-1 , r=1,\ldots,R$ and $\|\mathbf{A}^{(N)}\| \to +\infty$ when $||(\mathbf{A}^{(1)}, \cdots, \mathbf{A}^{(N)})||\to +\infty$. 
   This implies there exists ${\bar{j}_N,\bar{r} \in N}$ and  ${\bar{j}_i \in N} $ such that $|{a}^{(N)}_{{\bar{j}_N\bar{r}}}|\to +\infty$ and  ${a}^{(i)}_{{\bar{j}_i\bar{r}}}\neq 0$.
   Hence, for a fixed $\mathcal{X}$, we have if $||(\mathbf{A}^{(1)}, \cdots, \mathbf{A}^{(N)})||\to +\infty$, then $|{{X}}_{\bar{j}_1\cdots \bar{j}_N}-{\sum_{r=1}^{R}}{a}^{(1)}_{\bar{j}_1 r}\cdots {a}^{(N)}_{\bar{j}_N r}|^2\to  +\infty$.
   Combing with $0\leq\|\mathbf{A}^{(N)}\|_{2,0} \leq R$, we can obtain $F(\mathbf{A}^{(1)}, \cdots, \mathbf{A}^{(N)}) \to +\infty$.

   In conclusion, if $| |(\mathbf{A}^{(1)}, \cdots, \mathbf{A}^{(N)})||\to  +\infty$, then $F(\mathbf{A}^{(1)}, \cdots, \mathbf{A}^{(N)}) \to +\infty$. This implies that $F$ in \eqref{modelF} is coercive.
	Note that the objective function $F$ is a closed
    proper function, we have the set of optimal solutions of  \eqref{modelF}  is nonempty by Weierstrass’ Theorem. 
    By the equivalence between \eqref{model} and \eqref{modelF}, the set of optimal solutions of the CPD\_GSU model  \eqref{model}  is nonempty.
    This completes the proof.
\end{proof}

\section{Double-loop block-coordinate descent algorithms}
\label{sec4}
\indent
\par
This section aims to solve problem \eqref{modelF} via a double-loop block-coordinate descent proximal gradient descent algorithm together with a rank reduction strategy to reduce the computational complexity. 
We first present the algorithm in Subsection \ref{subsection41} in a more general problem setting, and then apply it with a rank reduction strategy in Subsection \ref{subsection42}.

\subsection{A double-loop block-coordinate descent proximal gradient descent algorithm}\label{subsection41}
\indent
\par
% \blue{
In this subsection, we present a double-loop block-coordinate descent proximal gradient descent algorithm for solving the unconstrained optimization problem 
\begin{equation}\label{m11}
    \min _{\mathbf{X}_1,\cdots,\mathbf{X}_N} F(\mathbf{X})=f(\mathbf{X}_1,\cdots,\mathbf{X}_N)+\sum\limits_{i=1}^N r_i(\mathbf{X}_i),
\end{equation}
where 
$\mathbf{X}=(\mathbf{X}_1,\cdots,\mathbf{X}_N )$ with 
each $\mathbf{X}_i=[\mathbf{x}_{i,1},\cdots,\mathbf{x}_{i, R_i}]$ being a group of vectors in (finite-dimensional) Euclidean spaces $\cX_{i,1}\ldots\cX_{i, R_i}$,   
$\cX_i:=\cX_{i,1}\times\cdots\times \cX_{i,R_i}$, 
$\cX:=\cX_i\times\cdots\times\cX_N$, 
$f:\cX\to\Re$ is a smooth (not necessarily convex) function, 
$r_i$ is a proper lower semicontinuous (not necessarily convex) function with the block-separable instruction that $r_i(\mathbf{X}_i)=\sum_{j=1}^{R_i}r_{i,j}(\mathbf{x}_{i,j})$.
Problem \eqref{m11} is the general form of problem \eqref{modelF}.
% for computing the \red{CP tensor rank problem??}
The algorithm given as Algorithm \ref{algorithm APGBCD1} for solving problem \eqref{m11} is a double-loop block-coordinate proximal gradient descent algorithm with extrapolation in which the inner subproblems 
are solved approximately via 
running Algorithm \ref{algorithm for GR1} for finitely many cycles. 
Here, $m_i$ is the number of inner-iterations for solving a subproblem with respect to the $i$-th block-variable $\mathbf{X}_i$.

Before presenting the algorithms, some notations are introduced for convenience.  
Denote $\mathbf{x}_{i,j}$ is the $j$-th block of $\mathbf{X}_i$,
$\mathbf{X}_i^k$ is $\mathbf{X}_i$ in the $k$th outer-iteration, and $\mathbf{x}_{i,j}^{k.v}$ is $\mathbf{x}_{i,j}$ in the $k$th outer-iteration and $v$-th inner-iteration.
For convenience, denote
\begin{equation}\nonumber
\begin{array}{ll}
     \mathbf{X}_{i,<j}:=(\mathbf{x}_{i,1},\cdots,\mathbf{x}_{i,j-1}),\quad 
     \mathbf{X}_{i,\geq j}:=(\mathbf{x}_{i,j},\cdots,\mathbf{x}_{i,R_i}).
\end{array}
\end{equation}

\begin{algorithm}
\caption{A double-loop block-coordinate  proximal gradient  descent algorithm with extrapolation for solving \eqref{m11}}
\label{algorithm APGBCD1}
\begin{algorithmic}
\State {\bf Initialization.}{
Given the initial point $\mathbf{X}^{0}$. 
Let $\{w^k\}$ be a non-decreasing positive sequence with $\gamma=\sup\{w^k\}<1$. 
Set the integers $m_1,\ldots,m_N>0$  and the real parameter $\varepsilon>0$. 
Let $\overline{\mathbf{X}}^0=\mathbf{X}^0$, $m=\max\{m_1,\ldots, m_N\}$ and $\tau=\frac{\varepsilon}{m}$.}
\State {1.} 
{\bf for} {$k=0,1,\ldots$}
\State {2.} {$\quad$ {\bf for} $i=1:N$}
\State {3.} {$\quad$ $\quad$ Update $\mathbf{X}_{i}^{k+1}={\rm sub\_BC\_PGD}({\mathbf{X}}_{i}^{k},  \widetilde{\mathbf{X}}_{<i}^{k+1},\overline{\mathbf{X}}_{>i}^{k},m_i,\varepsilon)$ by Algorithm \ref{algorithm for GR1}.}
\State {4.} {$\quad$ $\quad$ Update
		$\widetilde{\mathbf{X}}_{i}^{k+1} =	\mathbf{X}_{i}^{k+1}+w_k(\mathbf{X}_{i}^{k+1}-\overline{\mathbf{X}}_{i}^{k})$.}
		\State {5.} {$\quad$ {\bf end for}}
\State {6.} {$\quad$  {\bf if} $F(\mathbf{X}^{k+1})> F(\mathbf{X}^{k}) -\frac{\tau}{2}\|\mathbf{X}^{k+1}-\mathbf{X}^k\|^2$ }
\State {7.} {$\quad\quad$ {\bf for} $i=1:N$}
\State {8.} {$\quad\quad\quad$ Update $\mathbf{X}_{i}^{k+1}={\rm sub\_BC\_PGD}({\mathbf{X}}_{i}^{k},  {\mathbf{X}}_{<i}^{k+1}, {\mathbf{X}}_{>i}^{k},m_i,\varepsilon)$ by Algorithm \ref{algorithm for GR1}}.
\State {9.} {$\quad\quad\quad$ Update
		$\overline{\mathbf{X}}_{i}^{k+1} =	\mathbf{X}_{i}^{k+1}+w_k(\mathbf{X}_{i}^{k+1}-\overline{\mathbf{X}}_{i}^{k})$.}
\State {10.} {$\quad\quad$ {\bf end for}}
\State {11.} {$\ \  $  {\bf else}}
\State {12.}{ $\quad \quad$ Update
$\overline{\mathbf{X}}^{k+1}
=\widetilde{\mathbf{X}}^{k+1}$.}
\State {13.} {$\ \  $  {\bf end if}}
\State {14.} {\bf end for} 
\State {15.} {{\bf return} $\mathbf{X}^*=(\mathbf{X}_{i}^{k+1},\ldots,\mathbf{X}_{i}^{k+1})$.}
\end{algorithmic}
\end{algorithm}

\begin{algorithm}[H]
  \caption{The inner block-coordinate proximal gradient descent algorithm for solving subproblems \eqref{sub11} of 
Algorithm \ref{algorithm APGBCD1}}
  \label{algorithm for GR1}
  \begin{algorithmic}
  \State{{\bf function}  sub\_BC\_PGD$( \mathbf{X}_i^k, m_i,\mathbf{X}^{k+1}_{<i}, \mathbf{X}^{k}_{>i},\varepsilon)$}
  \State {1.} {Set $\mathbf{X}_i^{k,0}=\mathbf{X}_i^{k}$.}
  \State {2.} {\bf for} {$v=0:m_i$}
  \State {3.} $\quad$  {\bf for}  $j=1: R_i$
   % \State{$\quad$ $\quad$  $l_{i}:=\mathbf{d}_{i}^\top\mathbf{d}_{i}$.}
  \State {4.}  {$\quad\quad$ 
$
\mathbf{x}_{i,j}^{k,v+1}\in \operatorname{prox}^{r_{i,j}}_{\frac{1}{\varepsilon+l_{i,j}^{k,v}}}(\mathbf{x}_{i,j}^{k,v}- 
\frac{1}{\varepsilon+l_{i,j}^{k,v}}\nabla_{\mathbf{x}_{i,j}} f(\mathbf{X}_{<i}^{k+1}, 
\mathbf{X}_{i,<j}^{k,v+1},\mathbf{X}_{i,\ge j}^{k,v}, \mathbf{X}_{>i}^{k}) ).
$}
  \State {5.} {$\quad$ {\bf end for}}
  % \State {6.} {$\quad$ Set $v=v+1$.}
  \State {6.} {\bf end for}
  \State {7.} {\bf return} $\mathbf{X}_i^{k+1} =\mathbf{X}_i^{k,m_i+1}=(\mathbf{x}^{k,m_i+1}_{i,1},\cdots,\mathbf{x}^{k,m_i+1}_{i,R_i})$.
  \end{algorithmic}
  \end{algorithm}
  
\begin{remark}
    In Algorithm \ref{algorithm for GR1}, we use the block-coordinate proximal gradient descent algorithm to solve the following subproblems:
    \begin{equation} \label{sub11}
        \min_{\mathbf{X}_i} f(\mathbf{X}_{<i},\mathbf{X}_i,\mathbf{X}_{>i})+ r_i(\mathbf{X}_i), \quad i=1,\cdots,N.
    \end{equation}
    Each block is as follows:
    \begin{equation} \nonumber
    \begin{array}{ll}
    \mathbf{x}_{i,j} \in &\arg\min_{\mathbf{x}_{i,j}} \{
 \langle \nabla_{\mathbf{x}_{i,j}} f(\mathbf{X}_{<i}, 
\mathbf{X}_{i,<j}^{v+1},\mathbf{X}_{i,\ge j}^{v}, \mathbf{X}_{>i}), \mathbf{x}_{i,j}-\mathbf{x}_{i,j}^{v}\rangle\\[1mm]
&+\frac{\varepsilon+l_{i,j}^{v}}{2}\| \mathbf{x}_{i,j}-\mathbf{x}_{i,j}^{v}\|^2+r_{i,j}(\mathbf{x}_{i,j})\}.
    \end{array}
        \end{equation}
\end{remark}
\begin{remark}
    It is worth noting that, in our proposed algorithm, no additional complex parameters need to be selected beyond adjusting the regularization parameter $\lambda$ and the number of inner iterations.
\end{remark}

Subsequently, we use Algorithm \ref{algorithm APGBCD1} to solve the model \eqref{modelF}. Let $\mathbf{X}_i:=\mathbf{A}^{(i)}$, $\mathbf{x}_{i,j}:=\mathbf{a}^{(i)}_j$,
we have 
  \begin{equation}\nonumber
      \nabla_{\mathbf{a}_j^{(i)}} f(\mathbf{A}^{(1)},\cdots,\mathbf{A}^{(N)})=(\mathbf{a}_j^{(i)}\mathbf{d}_j^\top+\sum\limits_{r=1,r\neq j}^{R}\mathbf{a}_r^{{(i)}}\mathbf{d}_r^\top
	-\mathbf{X}_{(i)})\mathbf{d}_j,
  \end{equation}
  where $\mathbf{D}=(\mathbf{d}_1,\cdots,\mathbf{d}_R)=\mathbf{A}^{(N)}\odot\cdots\odot\mathbf{A}^{(i+1)}\odot\mathbf{A}^{(i-1)}\odot\cdots\odot\mathbf{A}^{(1)}$, $\mathbf{X}_{(i)}$ represents the mode-$i$ unfolding of the given tensor $\mathcal{X}$, and the Lipschitz constant of $\nabla_{\mathbf{a}_j^{(i)}} f(\mathbf{A}^{(1)},\cdots,\mathbf{A}^{(N)})$ is $\mathbf{d}^\top_j\mathbf{d}_j$.
Note that we have $\delta_{\mathbb{A}^{(i)}}(\mathbf{A}^{(i)})=\sum\limits_{r=1}^{R}\delta_{\mathbb{A}^{(i)}_r}(\mathbf{a}^{(i)}_r)$ with $\mathbb{A}^{(i)}_r=\{\mathbf{a}^{(i)}_r\in\mathds{R}^{n_i}\vert \|\mathbf{a}^{(i)}_r\|=1\}$.
The proximal mapping \cite[Definition 1.22]{rockafellar2009variational} of $\delta_{\mathbb{A}^{(i)}_r}(\mathbf{a}^{(i)}_r)$ is as follows:
$$
    {\rm prox}^{\delta_{\mathbb{A}^{(i)}_r}(\cdot)}_1(\mathbf{a}^{(i)}_r)= 
	\begin{cases}\mathbf{e},  &\|\mathbf{a}^{(i)}_r\|= 0, \\
		\frac{\mathbf{a}^{(i)}_r}{\|\mathbf{a}^{(i)}_r\|}, &\text{ otherwise. }\end{cases} 
$$
Here $\mathbf{e}$ is any vector that satisfies $\|\mathbf{e}\|=1$.
Moreover, we have $\lambda\|\mathbf{A}^{(N)}\|_{2,0}=\sum\limits_{r=1}^{R}\lambda\|\mathbf{a}^{(N)}_{r}\|^0$.
The proximal mapping of $\lambda\|\mathbf{a}^{(N)}_{r}\|^0$ is as follows:
$$
\begin{array}{lll}
	{\rm prox}_{\lambda}^{\|\cdot\|^0}(\mathbf{a}^{(N)}_{r})= \begin{cases}\mathbf{a}^{(N)}_{r}, &   \|\mathbf{a}^{(N)}_{r} \|>\sqrt{2 \lambda},\\
		\mathbf{0} \ or \ \mathbf{a}^{(N)}_{r}, &  \|\mathbf{a}^{(N)}_{r} \|=\sqrt{2 \lambda},  \\
		\mathbf{0}, &  \|\mathbf{a}^{(N)}_{r} \|<\sqrt{2 \lambda}.\end{cases} 
\end{array}
$$

\subsection{Convergence analysis of Algorithm \ref{algorithm APGBCD1}}
\indent
\par
This subsection presents the convergence analysis of Algorithm \ref{algorithm APGBCD1}.
Before that, we present the following assumptions.
\begin{assumption}\label{ASS}
\begin{itemize}
\item[(i)] $\inf_{\mathcal{X}_i\times\cdots\times\mathcal{X}_N} f>-\infty$, $\inf_{\mathcal{X}_{i}} r_i>-\infty,i=1,\cdots,N$ and the proximal mapping of each $r_i$ can be explicitly calculated. 
\item [(ii)] For any fixed $\mathbf{X}_p, p\neq i$, the partial gradient $\nabla_{\mathbf{X}_i} f(\mathbf{X}_1, \cdots, \mathbf{X}_N)$ is globally Lipschitz with moduli $l_{\mathbf{X}_{p\neq i}}$.
\item [(iii)] For $i=1,\cdots,N$, there exists $l_i^+$ such that 
$\sup\{l_{\mathbf{X}^k_{p\neq i}}: k\in\mathbb{N}\}\leq l_i^+$.
 \item [(iv)] $\nabla f(\mathbf{X})$
is Lipschitz continuous on bounded subsets of $\mathcal{X}_i\times\cdots\times\mathcal{X}_N$, i.e., for each bounded subset $\mathbb{S}$ of $\mathcal{X}_i\times\cdots\times\mathcal{X}_N$, there exists $M>0$, such that for 
all $\mathbf{X}^i=(\mathbf{X}_1^i,\cdots,\mathbf{X}_N^i)\in \mathbb{S}, i=1,2$,
$ \|\nabla f(\mathbf{X}^1)-\nabla f(\mathbf{X}^2) \|\leq M \| \mathbf{X}^1-\mathbf{X}^2 \|$.
\end{itemize}
\end{assumption}

The following result is a descent lemma for Algorithm \ref{algorithm APGBCD1}. 
\begin{lemma}\label{sd1}
Suppose that Assumption \ref{ASS} holds.
Let the sequence $\{ {\mathbf{X}}^k\}$ be generated by Algorithm \ref{algorithm APGBCD1} from the starting point ${\mathbf{X}}^0$. 
The following assertions hold:
\begin{description}
\item[(i)] The sequence $\{ F({\mathbf{X}}^k)\}$ is non-increasing, and one has
\begin{equation}
\label{sufficientdescent}
F(\mathbf{X}^{k+1})+\frac{\tau}{2} \|\mathbf{X}^{k+1}-\mathbf{X}^k \|^2 \leq F(\mathbf{X}^k),\quad  \forall k \geq 0,
\end{equation}
\item[(ii)] It holds that 
$\lim_{k \to \infty}\|\mathbf{X}^{k+1}-\mathbf{X}^k\|=0$ and $\lim_{k \to \infty} \big\|\overline{\mathbf{X}}^{k}-\mathbf{X}^k \big\|=0$.
\end{description}
\end{lemma}
\begin{proof}
{\bf (i)}
At the $k$-th iteration of Algorithm  \ref{algorithm APGBCD1}, if Step 9 is skipped, we know that   
\eqref{sufficientdescent} holds automatically. 
On the other hand, if Step 9 is executed, it is a basic descent property of proximal gradient steps 
(c.f. \cite[Lemma 3]{Bottle} or \cite[Lemmas 3.2 \& 3.3]{WangCOAP}) that   
$$
F (\mathbf{X}^{k+1}_{<i}, \mathbf{X}^{k,v}_{i}, \mathbf{X}^k_{>i})
\ge 
\frac{\varepsilon}{2} \|\mathbf{X}_{i}^{k,v}-\mathbf{X}_{i}^{k,v+1}\|^2
+
F(\mathbf{X}^{k+1}_{<i}, \mathbf{X}^{k,v+1}_{i}, \mathbf{X}^k_{>i}).
$$
Therefore,  
$
F (\mathbf{X}^{k+1}_{<i},  \mathbf{X}^k_{\ge i})
\ge 
\frac{\tau}{2} \|\mathbf{X}_{i}^{k}-\mathbf{X}_{i}^{k+1} \|^2
+
F(\mathbf{X}^{k+1}_{\le i}, \mathbf{X}^k_{>i})$, which implies \eqref{sufficientdescent}.

{\bf (ii)} 
Since $F$ is bounded from below, one has from \eqref{sufficientdescent} that 
$\sum\limits_{k=0}^{+\infty} \|\mathbf{X}^k-{\mathbf{X}}^{k+1} \|^2 < +\infty$,  so that 
$\lim_{k \to \infty}\|\mathbf{X}^{k+1}-\mathbf{X}^k\|=0$.
Note that by step 5 of Algorithm \ref{algorithm APGBCD1} one has 
$$\overline{\mathbf{X}}_i^{k+1}-\mathbf{X}^{k+1}_i
=w_k (\mathbf{X}^{k+1}_i-\mathbf{X}^k_i )+w_k (\mathbf{X}^k_i-\overline{\mathbf{X}}^k_i ).$$ 
Therefore, one has
$$
\begin{array}{llll}
\|\mathbf{X}_i^{k+1}-\mathbf{X}_i^k \|^2 
\\
=\frac{1}{ w_k^2} \|\overline{\mathbf{X}}_i^{k+1}-\mathbf{X}_i^{k+1} \|^2+ \|\mathbf{X}_i^k-\overline{\mathbf{X}}_i^k \|^2
-\frac{2}{w_k} \langle\overline{\mathbf{X}}_i^{k+1}-\mathbf{X}_i^{k+1}, \mathbf{X}_i^k-\overline{\mathbf{X}}_i^k \rangle 
\\
\geq \frac{1}{ w_k^2} \|\overline{\mathbf{X}}_i^{k+1}-\mathbf{X}_i^{k+1} \|^2+ \|\mathbf{X}_i^k-\overline{\mathbf{X}}_i^k \|^2
-\frac{1}{w_k} ( \|\overline{\mathbf{X}}_i^{k+1}-\mathbf{X}_i^{k+1} \|^2+ \|\mathbf{X}_i^k-\overline{\mathbf{X}}_i^k \|^2 ) \\
= (\frac{1}{ w_k^2}-\frac{1}{w_k} ) \|\overline{\mathbf{X}}_i^{k+1}-\mathbf{X}_i^{k+1} \|^2+ (1-\frac{1}{w_k} ) \|\overline{\mathbf{X}}_i^k-\mathbf{X}_i^k \|^2.
\end{array}
$$
Since $\{w^k\}$ is non-decreasing and $\lim\limits_{k\to\infty} w^k=\gamma<1$, one can see that
$$
\begin{array}{lll}
		+\infty & > \sum\limits_{k=0}^{+\infty} \|\mathbf{X}_i^{k+1}-{\mathbf{X}_i}^k \|^2  \\
  		& \geq \sum\limits_{k=1}^{+\infty} (1+\frac{1}{ w_{k-1}^2}-\frac{1}{w_k}-\frac{1}{w_{k-1}} ) \|\overline{\mathbf{X}}_i^k-\mathbf{X}_i^k \|^2-  (1-\frac{1}{w_0} ) \|\overline{\mathbf{X}}_i^0-\mathbf{X}_i^0 \|^2\\
    		& \geq \sum\limits_{k=1}^{+\infty} (1+\frac{1}{ w_k^2}-\frac{2}{w_k} ) \|\overline{\mathbf{X}}_i^k-\mathbf{X}_i^k \|^2-  (1-\frac{1}{w_0} )\|\overline{\mathbf{X}}_i^0-\mathbf{X}_i^0 \|^2\\
		&=\sum\limits_{k=1}^{+\infty} (1-\frac{1}{w_k} )^2 \|\overline{\mathbf{X}}_i^k-\mathbf{X}_i^k \|^2
		\geq  (1-\frac{1}{\gamma} )^2\sum\limits_{k=1}^{+\infty} \|\overline{\mathbf{X}}_i^k-\mathbf{X}_i^k \|^2,
	\end{array}
$$
which implies $\lim _{k \to +\infty} \|\overline{\mathbf{X}}^k-\mathbf{X}^{k} \|=0$.
The conclusions are obtained. 
\end{proof}

Next, we establish the convergence theorem for Algorithm \ref{algorithm APGBCD1}.
\begin{theorem}
Suppose that Assumption \ref{ASS} holds. 
Let $\{{\mathbf{X}}^{k}\}_{k\in \mathbb{N}}$ be the sequence generated by Algorithm \ref{algorithm APGBCD1} from the starting point ${\mathbf{X}}^0$. 
Then, any accumulation point of  $\{ {\mathbf{X}}^k\}_{k\in\mathbb{N}}$  is a stationary point of problem \eqref{m11}.
\end{theorem}

\begin{proof}
    We first present the relative error bound of the subgradients. 
Given  $k\ge 0$ and any $1\le i\le N$, suppose that 
$\mathbf{X}^{k+1}$ is calculated by step 3. 
By applying Fermat's rule \cite[Theorem 10.1]{rockafellar2009variational} to the optimization problem in step 3 and using \cite[Proposition 10.5 \& Exercise 10.10]{rockafellar2009variational}, one has for any $1\le j\le R_i$ there exists a vector $\mathbf{g}_i^j \in \partial r_{i}^j(\mathbf{X}_i^{k+1})$ such that
$$
\nabla_{\mathbf{x}_{i,j}} f(\overline{\mathbf{X}}^{k+1}_{<i}, \mathbf{X}_{i, < j}^{k,m_i+1}, \mathbf{X}_{i,\geq j}^{k,m_i}, \overline{\mathbf{X}}^{k+1}_{>i})+(l_{i,j}^{k,m_i}+\varepsilon)(\mathbf{x}_{i,j}^{k, m_i+1}-\mathbf{x}_{i,j}^{k, m_i})+\mathbf{g}_i^j=0. 
$$
From the Lipschitz continuity of the partial gradient, we know that 
$$
\begin{array}{ll}
&\|
\nabla_{\mathbf{x}_{i,j}} f(\overline{\mathbf{X}}^{k+1}_{<i}, \mathbf{X}_{i, < j}^{k,m_i+1}, \mathbf{X}_{i,\geq j}^{k,m_i}, \overline{\mathbf{X}}^{k+1}_{>i})
-
\nabla_{\mathbf{x}_{i,j}} f(\overline{\mathbf{X}}^{k+1}_{<i}, {\mathbf{X}}^{k+1}_{i}, \overline{\mathbf{X}}^{k}_{>i})\|
\\[1mm]
&\le
l_i^+ \|\mathbf{X}^{k+1}_{i}-\mathbf{X}_i^{k, m_i}\|.
\end{array}
$$
From the above equality and inequality, one has 
there exist $\mathbf{g}_i =(\mathbf{g}_i^1,\cdots,\mathbf{g}_i^{R_i}) \in \partial r_{i}(\mathbf{X}_i^{k+1})$ such that 
$$
\begin{array}{ll}
\|\nabla_{\mathbf{X}_i} f(\overline{\mathbf{X}}^{k+1}_{<i}, {\mathbf{X}}^{k+1}_{i}, \overline{\mathbf{X}}^{k}_{>i})+\mathbf{g}_i \| \leq \mu_i \|\mathbf{X}^{k+1}_{i}-\mathbf{X}_i^{k, m_i} \|.
\end{array}
$$
Therefore, 
\begin{equation} \label{sga1+} 
\begin{array}{ll}
\|\nabla_{\mathbf{X}_i} f({\mathbf{X}}^{k+1})+\mathbf{g}_i \| 
\\
\leq \mu_i \|\mathbf{X}^{k+1}_{i}-\mathbf{X}_i^{k, m_i} \|
+
\|\nabla_{\mathbf{X}_i} f({\mathbf{X}}^{k+1})
-
\nabla_{\mathbf{X}_i} f(\overline{\mathbf{X}}^{k+1}_{<i}, {\mathbf{X}}^{k+1}_{i}, \overline{\mathbf{X}}^{k}_{>i})
\|\\[1mm]
\le \mu_i \|\mathbf{X}^{k+1}_{i}-\mathbf{X}_i^{k, m_i} \|+ \frac{M}{w_k^2}\|\overline{\mathbf{X}}^{k+1}-{\mathbf{X}}^{k+1}\|\\
\le \mu_i \|\mathbf{X}^{k+1}_{i}-\mathbf{X}_i^{k, m_i} \|+ \frac{M}{w_1^2}\|\overline{\mathbf{X}}^{k+1}-{\mathbf{X}}^{k+1}\|.
\end{array}
\end{equation}
Here, $\mathbf{X}_i^{k, m_i}=(\mathbf{x}_{i,1}^{k,m_i}, \cdots, \mathbf{x}_{i,R_i}^{k,m_i})$ and $\mu_i=R_i(\varepsilon+ 2l_i^+)$.

On the other hand, if if $\mathbf{X}^{k+1}$ is calculated by step 9,
by repeating the above procedure to the optimization problem in step 9
one can get that 
there exist $\mathbf{g}_i \in \partial r_{i}(\mathbf{X}_i^{k+1})$ such that 
 \begin{equation}
  \begin{array}{ll}\nonumber
	     \|\nabla_{\mathbf{X}_i} f({\mathbf{X}}^{k+1}_{<i}, {\mathbf{X}}^{k+1}_{i}, {\mathbf{X}}^{k}_{ >i})+\mathbf{g}_i \| &\leq \mu_i \|\mathbf{X}^{k+1}_{i}-\mathbf{X}_i^{k, m_i} \|.
      \end{array}
\end{equation}
Therefore, 
\begin{equation} \label{sga2+} 
\begin{array}{ll}
&\|\nabla_{\mathbf{X}_i} f({\mathbf{X}}^{k+1})+\mathbf{g}_i \| \\[1mm]
&\leq \mu_i \|\mathbf{X}^{k+1}_{i}-\mathbf{X}_i^{k, m_i} \|
+
\|\nabla_{\mathbf{X}_i} f({\mathbf{X}}^{k+1})
-
\nabla_{\mathbf{X}_i} f({\mathbf{X}}^{k+1}_{<i}, {\mathbf{X}}^{k+1}_{i}, {\mathbf{X}}^{k}_{ >i})
\|\\[1mm]
&\le \mu_i \|\mathbf{X}^{k+1}_{i}-\mathbf{X}_i^{k, m_i} \|+M\|{\mathbf{X}}^{k+1}-{\mathbf{X}}^{k}\|\\[1mm]
&\le \mu_i \|\mathbf{X}^{k+1}_{i}-\mathbf{X}_i^{k, m_i} \|+\frac{2M}{w_1^2}\|\overline{\mathbf{X}}^{k+1}-\overline{\mathbf{X}}^{k+1}\|+2M\|\overline{\mathbf{X}}^{k}-\overline{\mathbf{X}}^{k}\|.
\end{array}
\end{equation}
Combining \eqref{sga1+} and \eqref{sga2+}, we can deduce that
\begin{equation} \label{sga3+} 
\begin{array}{ll}
\|\nabla_{\mathbf{X}_i} f({\mathbf{X}}^{k+1})+\mathbf{g}_i \| 
\le \mu_i \|\mathbf{X}^{k+1}_{i}-\mathbf{X}_i^{k, m_i} \|+\frac{2M}{w_1^2}\|\overline{\mathbf{X}}^{k+1}-\overline{\mathbf{X}}^{k+1}\|+2M\|\overline{\mathbf{X}}^{k}-\overline{\mathbf{X}}^{k}\|.
\end{array}
\end{equation}

Consider the upper bound of $ \|\mathbf{X}^{k+1}_{i}-\mathbf{X}_i^{k, m_i} \|$.
Suppose that 
$\mathbf{X}^{k+1}$ is calculated by step 3, we have 
 \begin{equation}
\begin{array}{lll}\label{FFx1}
&  \|\mathbf{X}_i^{k+1}-\mathbf{X}_i^{k,m_i} \|^2\leq \frac{2 }{\varepsilon} (F (\overline{\mathbf{X}}_{<i}^{k+1},{\mathbf{X}}_i^{k,m_i},\overline{\mathbf{X}}_{>i}^k) - F (\overline{\mathbf{X}}_{<i}^{k+1},{\mathbf{X}}_i^{k+1},\overline{\mathbf{X}}_{>i}^k))\\[1mm]
& \leq \frac{2 }{\varepsilon} (F (\overline{\mathbf{X}}_{<i}^{k+1},{\mathbf{X}}_i^{k},\overline{\mathbf{X}}_{>i}^k) - F (\overline{\mathbf{X}}_{<i}^{k+1},{\mathbf{X}}_i^{k+1},\overline{\mathbf{X}}_{>i}^k))\\[1mm]
& = \frac{2 }{\varepsilon} (f (\overline{\mathbf{X}}_{<i}^{k+1},{\mathbf{X}}_i^{k},\overline{\mathbf{X}}_{>i}^k) - f (\overline{\mathbf{X}}_{<i}^{k+1},{\mathbf{X}}_i^{k+1},\overline{\mathbf{X}}_{>i}^k))+ \frac{2 }{\varepsilon}(r_i(\mathbf{X}_i^k)-r_i(\mathbf{X}_i^{k+1}))\\[1mm]
&= \frac{2 }{\varepsilon} (F (\mathbf{X}_{<i}^{k+1},{\mathbf{X}}_{\geq i}^k) - F (\mathbf{X}_{\leq i}^{k+1},{\mathbf{X}}_{> i}^k)+f (\mathbf{X}_{\leq i}^{k+1},{\mathbf{X}}_{> i}^k) \\[1mm]
& - f (\mathbf{X}_{<i}^{k+1},{\mathbf{X}}_{\geq i}^k)+ f (\overline{\mathbf{X}}_{<i}^{k+1},{\mathbf{X}}_i^{k},\overline{\mathbf{X}}_{>i}^k) - f (\overline{\mathbf{X}}_{<i}^{k+1},{\mathbf{X}}_i^{k+1},\overline{\mathbf{X}}_{>i}^k))\\[1mm]
&\leq \frac{2 }{\varepsilon} (F (\mathbf{X}_{<i}^{k+1},{\mathbf{X}}_{\geq i}^k) - F (\mathbf{X}_{\leq i}^{k+1},{\mathbf{X}}_{> i}^k) + l_i^+ \|\mathbf{X}_i^{k+1}-{\mathbf{X}}_i^{k} \|^2 \\[1mm]
& +  \langle \mathbf{X}_i^{k+1}-\mathbf{X}_i^k, \nabla_{\mathbf{X}_i} f(\mathbf{X}_{<i}^{k+1},{\mathbf{X}}_{\geq i}^k) -\nabla_{\mathbf{X}_i} f(\overline{\mathbf{X}}_{<i}^{k+1},{\mathbf{X}}_i^{k+1},\overline{\mathbf{X}}_{>i}^k)) \rangle)\\[1mm]
& \leq \frac{2 }{\varepsilon} (F (\mathbf{X}_{<i}^{k+1},{\mathbf{X}}_{\geq i}^k) - F (\mathbf{X}_{\leq i}^{k+1},{\mathbf{X}}_{> i}^k)+(l_i^++\frac{M^2+1}{2}) \|\mathbf{X}_i^{k+1}-\mathbf{X}_i^k \|^2 \\[1mm]
& +\frac{M^2}{2}(\sum\limits_{p=1}^{i-1} \|\overline{\mathbf{X}}_p^{k+1}-{\mathbf{X}}_p^{k+1} \|^2+\sum\limits_{p=i+1}^{N} \|\overline{\mathbf{X}}_p^{k}-{\mathbf{X}}_p^k \|^2)).
\end{array}
\end{equation}

Suppose that 
$\mathbf{X}^{k+1}$ is calculated by step 9, we have 
 \begin{equation}
\begin{array}{lll}\label{FFx2}
&  \|\mathbf{X}_i^{k+1}-\mathbf{X}_i^{k,m_i} \|^2\leq \frac{2 }{\varepsilon} (F ({\mathbf{X}}_{<i}^{k+1},{\mathbf{X}}_i^{k,m_i},{\mathbf{X}}_{>i}^k) - F ({\mathbf{X}}_{<i}^{k+1},{\mathbf{X}}_i^{k+1},{\mathbf{X}}_{>i}^k))\\[1mm]
& \leq \frac{2 }{\varepsilon} (F ({\mathbf{X}}_{<i}^{k+1},{\mathbf{X}}_i^{k},{\mathbf{X}}_{>i}^k) - F ({\mathbf{X}}_{<i}^{k+1},{\mathbf{X}}_i^{k+1},{\mathbf{X}}_{>i}^k))\\[1mm]
& = \frac{2 }{\varepsilon} (f ({\mathbf{X}}_{<i}^{k+1},{\mathbf{X}}_i^{k},{\mathbf{X}}_{>i}^k) - f ({\mathbf{X}}_{<i}^{k+1},{\mathbf{X}}_i^{k+1},{\mathbf{X}}_{>i}^k))+ \frac{2 }{\varepsilon}(r_i(\mathbf{X}_i^k)-r_i(\mathbf{X}_i^{k+1}))\\[1mm]
&= \frac{2 }{\varepsilon} (F (\mathbf{X}_{<i}^{k+1},{\mathbf{X}}_{\geq i}^k) - F (\mathbf{X}_{\leq i}^{k+1},{\mathbf{X}}_{> i}^k)+f (\mathbf{X}_{\leq i}^{k+1},{\mathbf{X}}_{> i}^k) \\[1mm]
& - f (\mathbf{X}_{<i}^{k+1},{\mathbf{X}}_{\geq i}^k)+ f ({\mathbf{X}}_{<i}^{k+1},{\mathbf{X}}_i^{k},{\mathbf{X}}_{>i}^k) - f ({\mathbf{X}}_{<i}^{k+1},{\mathbf{X}}_i^{k+1},{\mathbf{X}}_{>i}^k))\\[1mm]
&\leq \frac{2 }{\varepsilon} (F (\mathbf{X}_{<i}^{k+1},{\mathbf{X}}_{\geq i}^k) - F (\mathbf{X}_{\leq i}^{k+1},{\mathbf{X}}_{> i}^k) + l_i^+ \|\mathbf{X}_i^{k+1}-{\mathbf{X}}_i^{k} \|^2 \\[1mm]
& +  \langle \mathbf{X}_i^{k+1}-\mathbf{X}_i^k, \nabla_{\mathbf{X}_i} f(\mathbf{X}_{<i}^{k+1},{\mathbf{X}}_{\geq i}^k) -\nabla_{\mathbf{X}_i} f({\mathbf{X}}_{<i}^{k+1},{\mathbf{X}}_i^{k+1},{\mathbf{X}}_{>i}^k)) \rangle)\\[1mm]
& \leq \frac{2 }{\varepsilon} (F (\mathbf{X}_{<i}^{k+1},{\mathbf{X}}_{\geq i}^k) - F (\mathbf{X}_{\leq i}^{k+1},{\mathbf{X}}_{> i}^k)+(l_i^++\frac{M^2+1}{2}) \|\mathbf{X}_i^{k+1}-\mathbf{X}_i^k \|^2 \\[1mm]
& +\frac{M^2}{2}\sum\limits_{p=i+1}^{N} \|{\mathbf{X}}_p^{k+1}-{\mathbf{X}}_p^k \|^2).
\end{array}
\end{equation}
Combing \eqref{FFx1} and \eqref{FFx2}, we can deduce that
\begin{equation}
\begin{array}{lll}\nonumber
\|\mathbf{X}_i^{k+1}-\mathbf{X}_i^{k,m_i} \|^2&\leq \frac{2 }{\varepsilon} (F (\mathbf{X}_{<i}^{k+1},{\mathbf{X}}_{\geq i}^k) - F (\mathbf{X}_{\leq i}^{k+1},{\mathbf{X}}_{> i}^k)\\[1mm]
&+(l_i^++\frac{M^2+1}{2}) \|\mathbf{X}^{k+1}-\mathbf{X}^k \|^2\\[1mm]
&+\frac{M^2}{2}(\|\overline{\mathbf{X}}^{k+1}-{\mathbf{X}}^{k+1} \|^2+\|\overline{\mathbf{X}}^{k}-{\mathbf{X}}^k \|^2)).
\end{array}
\end{equation}
Therefore, we can derive 
\begin{equation}
\begin{array}{lll}\nonumber
&\|\mathbf{X}^{k+1}-\mathbf{X}^{k,m_i} \|^2=\sum\limits_{i=1}^N\|\mathbf{X}_i^{k+1}-\mathbf{X}_i^{k,m_i} \|^2\\[1mm]
&\leq \sum\limits_{i=1}^{N} (\frac{2 }{\varepsilon} (F (\mathbf{X}_{<i}^{k+1},{\mathbf{X}}_{\geq i}^k) - F (\mathbf{X}_{\leq i}^{k+1},{\mathbf{X}}_{> i}^k)) \\[1mm]
&+ N(l_i^++\frac{M^2+1}{2}) \|\mathbf{X}^{k+1}-\mathbf{X}^k \|^2+\frac{N M^2}{2}(\|\overline{\mathbf{X}}^{k+1}-{\mathbf{X}}^{k+1} \|^2+\|\overline{\mathbf{X}}^{k}-{\mathbf{X}}^k \|^2))\\[1mm]
&=\frac{2}{\varepsilon}(F (\mathbf{X}^{k}) - F (\mathbf{X}^{k+1}))+ N(l_i^++\frac{M^2+1}{2}) \|\mathbf{X}^{k+1}-\mathbf{X}^k \|^2\\[1mm]
&+\frac{N M^2}{2}(\|\overline{\mathbf{X}}^{k+1}-{\mathbf{X}}^{k+1} \|^2+\|\overline{\mathbf{X}}^{k}-{\mathbf{X}}^k \|^2)).
\end{array}
\end{equation}

From Lemma \ref{sd1}, we have $F(\mathbf{X}^k)$ is non-increasing and bounded from below. Hence, $\lim_{k \to \infty}\|\mathbf{X}^{k+1}-\mathbf{X}^k\|=0$ and $\lim_{k \to \infty} \big\|\overline{\mathbf{X}}^{k}-\mathbf{X}^k \big\|=0$. Consequently, it can be deduced that $\lim_{k \to \infty}\|\mathbf{X}^{k+1}-\mathbf{X}^{k,m_i} \|=0$, which implies $\lim_{k \to \infty}\|\mathbf{X}^{k+1}_i-\mathbf{X}^{k,m_i}_i \|=0, i=1,\cdots,N$.
Therefore, we can derive $\lim_{k\to\infty}\|\nabla_{\mathbf{X}_i} f({\mathbf{X}}^{k+1})+\mathbf{g}_i\|=0$ by \eqref{sga3+}.

Using \cite[Proposition 10.5 \& Exercise 10.10]{rockafellar2009variational}, we have
$$\nabla_{\mathbf{X}} f({\mathbf{X}}^{k+1})+\mathbf{g}=(\nabla_{\mathbf{X}_1} f({\mathbf{X}}^{k+1}),\cdots,\nabla_{\mathbf{X}_N} f({\mathbf{X}}^{k+1}))+(\mathbf{g}_1,\cdots,\mathbf{g}_N) \in \partial_{\mathbf{X}} F({\mathbf{X}}^{k+1}).$$
Therefore, we can derive $\lim_{k\to\infty}\|\nabla_{\mathbf{X}} f({\mathbf{X}}^{k+1})+\mathbf{g}\|=0$.

Denote $\mathbf{X}^*$ to be an accumulation point of $\{\mathbf{X}^k\}_{k \in \mathbb{N}}$, there exists a subsequence $\{\mathbf{X}^{k_q}\}$ of $\{\mathbf{X}^k\}_{k \in \mathbb{N}}$ converging to $\mathbf{X}^*$. We can deduce that $\|\nabla_{\mathbf{X}} f({\mathbf{X}}^{*})+\mathbf{g}^*\|=\lim_{k\to\infty}\|\nabla_{\mathbf{X}} f({\mathbf{X}}^{k+1})+g\|=0$,
which implies
$0\in F({\mathbf{X}})^{*}$. This indicates that $\mathbf{X}^*$ is the stationary point of problem \eqref{m11}.
This completes the proof.
\end{proof}

Based on the analysis, we conclude that each accumulation point of the sequence generated by Algorithm \ref{algorithm APGBCD1} is a stationary point of  \eqref{m11}.

\subsection{Rank reduction strategy}\label{subsection42}
\indent
\par

In this subsection, we develop a rank reduction strategy and incorporate it into Algorithm \ref{algorithm APGBCD1} when solving \eqref{modelF} to reduce the computational complexity.

Before presenting the rank reduction strategy, we first prove that for a convergent sequence $\{(\mathbf{A}^{(1)^k}, \cdots, \mathbf{A}^{(N)^k})\}_{k\in \mathbb{N}}$,  the nonzero columns of $\mathbf{A}^{(N)^k}$ remain unchanged after a certain number of iterations in the following theorem.
\begin{theorem}\label{supp}
	Let $\{(\mathbf{A}^{(1)^k}, \cdots, \mathbf{A}^{(N)^k}) \}_{k\in \mathbb{N}}$ be the sequence that converges to $(\mathbf{A}^{(1)^*}, \cdots, \mathbf{A}^{(N)^*})$ for solving \eqref{modelF}.
    There exists a constant $M$ such that for all $k>M$, the nonzero columns of $\mathbf{A}^{(N)^k}$ remain unchanged and are the same as the nonzero columns of $\mathbf{A}^{(N)^*}$.
\end{theorem}
\begin{proof}
	For a sequence $ \{(\mathbf{A}^{(1)^k}, \cdots, \mathbf{A}^{(N)^k}) \}_{k\in \mathbb{N}}$ converges to $(\mathbf{A}^{(1)^*}, \cdots, \mathbf{A}^{(N)^*})$, we have $\lim_{k\to  \infty} \|\mathbf{A}^{(N)^k}-\mathbf{A}^{(N)^*}\|=0$.
	Hence, for any $\varepsilon < \lambda$, there exists $M$ such that for all $k>M$, the inequality $\|\mathbf{A}^{(N)^k}-\mathbf{A}^{(N)^*}\|\leq \varepsilon$ holds.
	
	Now we prove that for all $k>M$, the equality ${\rm supp}(\mathbf{A}^{(N)^k})={\rm supp}(\mathbf{A}^{(N)^*})$ holds by contradiction. 
    If ${\rm supp}(\mathbf{A}^{(N)^k}) \neq {\rm supp}(\mathbf{A}^{(N)^*})$, then there exists $j$
	such that $ \|\mathbf{a}^{(N)^k}_j \|>0$, $ \|\mathbf{a}^{(N)^*}_j \|=0$ or $ \|\mathbf{a}^{(N)^k}_j \|=0$, $ \|\mathbf{a}^{(N)^*}_j \|>0$.
    By $ \|\mathbf{A}^{(N)^k}-\mathbf{A}^{(N)^*} \|\leq \varepsilon$, we can deduce that $0< \|\mathbf{a}^{(N)^k}_j \|\leq\varepsilon$ if $ \|\mathbf{a}^{(N)^k}_j \|>0$, $ \|\mathbf{a}^{(N)^*}_j \|=0$, or $0< \|\mathbf{a}^{(N)^*}_j \|\leq \varepsilon$ if $ \|\mathbf{a}^{(N)^k}_j \|=0$, $ \|\mathbf{a}^{(N)^*}_j \|>0$.
    This contradicts the definition of $\operatorname{prox}_{\lambda}^{\|\cdot\|^{0}}(\cdot)$.
    Therefore, we have ${\rm supp}(\mathbf{A}^{(N)^k}) = {\rm supp}(\mathbf{A}^{(N)^*})$.
    This completes the proof.
\end{proof}
	
Based on Theorem \ref{supp}, we develop a rank reduction strategy as follows:
Assume that there exists a constant $M$ such that for all $k>M$, the nonzero columns of $\mathbf{A}^{(N)^k}$ remain unchanged.
Let $\mathbb{S}$ to be the index set of the nonzero columns of $\mathbf{A}^{(N)^k}$.
We remove the zero columns of $\mathbf{A}^{(N)^{M+1}}$ and the corresponding columns of $\mathbf{A}^{(i)^{M+1}}, i=1,\cdots, N-1$, i.e., $\mathbf{A}^{(i)^{M+2}}={\mathbf{A}}^{(i)^{M+1}}(:,\mathbb{S}), i=1,\cdots, N$, and then we set $R=|\mathbb{S}|$ in subsequent iterations. This strategy reduces computational complexity by eliminating unnecessary components while maintaining the accuracy of the CP decomposition.

We incorporate the proposed rank reduction strategy into Algorithm \ref{algorithm APGBCD1} to reduce the computational complexity when solving \eqref{modelF}.
Let $\mathbf{A}:=(\mathbf{A}^{(1)},\cdots,\mathbf{A}^{(N)})$.
The algorithm is summarized in Algorithm \ref{algorithm APGBCDRR}.
\begin{algorithm}[H]
\caption{The double-loop block-coordinate proximal gradient descent algorithm with rank reduction for \eqref{model} }
\label{algorithm APGBCDRR}
	\begin{algorithmic}
		\State {\bf Initialization.}{
  Given the initial point $\mathbf{A}^{0}$. 
Let $\{w^k\}$ be a non-decreasing positive sequence with $\gamma=\sup\{w^k\}<1$. 
Set the integers  $R>0$, $l=0$, $k=0$, $L>0$, $m_i>0, i=1,\cdots,N$, $m=\max\{m_i\}$  and the real parameter $\varepsilon>0$. 
Let $\overline{\mathbf{A}}^{0}=\mathbf{A}^{0}$ and $\tau=\frac{\varepsilon}{m}$.}
    \State {1.} {\bf while} {$l<L$}
    \State {2.} {$\quad$ {\bf for} $i=1:N$}
    \State {3.} $\quad$ $\quad$ Update {\small $\mathbf{A}^{(i)^{k+1}}={\rm sub\_BC\_PGD}(\mathbf{A}^{(i)^{k}},  \widetilde{\mathbf{A}}_{(<i)}^{{k+1}},\overline{\mathbf{A}}_{(>i)}^{{k}},m_i,\varepsilon)$} by Algorithm \ref{algorithm for GR1}.
    \State {4.} {$\quad$ $\quad$ Update
		$\widetilde{\mathbf{A}}^{(i)^{k+1}} =	\mathbf{A}^{(i)^{k+1}}+w_k(\mathbf{A}^{(i)^{k+1}}-\overline{\mathbf{A}}^{(i)^{k}})$.}
    \State {5.} {$\quad$ {\bf end for}}
    \State {6.} {$\quad$  {\bf if} {$F(\mathbf{A}^{k+1})> F(\mathbf{A}^{{k}}) -\frac{\tau}{2}\|\mathbf{A}^{k+1}-\mathbf{A}^{k}\|^2$} }
    \State {7.} {$\quad\quad$ {\bf for} $i=1:N$}
    \State {8.} {$\quad\quad\quad$ Update {\small $\mathbf{A}^{(i)^{k+1}}={\rm sub\_BC\_PGD}(\mathbf{A}^{(i)^{k}},{\mathbf{A}}_{(<i)}^{{k+1}},{\mathbf{A}}_{(>i)}^{{k}},m_i,\varepsilon)$} by Algorithm \ref{algorithm for GR1}.}
    \State {9.} $\quad\quad\quad$ Update $\overline{\mathbf{A}}^{(i)^{k+1}} =	\mathbf{A}^{(i)^{k+1}}+w_k(\mathbf{A}^{(i)^{k+1}}-\overline{\mathbf{A}}^{(i)^{k}})$.
    \State {10.} {$\quad\quad$ {\bf end for}}
    \State {11.} {$\ \  $  {\bf else}}
    \State {12.}{ $\quad \quad$ Update
    $\overline{\mathbf{A}}^{k+1}=\widetilde{\mathbf{A}}^{k+1}$.}
    \State {13.} {$\ \  $  {\bf end if}}
		\State {14.} {$\quad$ {\bf if} ${\rm supp}({\mathbf{A}}^{(N)^{k+1}})={\rm supp}({\mathbf{A}}^{(N)^{k}})$} 
            \State {15.} {$\quad\quad$ Set $l=l+1$.}
            \State {16.} {$\quad$ {\bf else} }
            \State {17.} {$\quad\quad$ Set $l=0$.}
		\State {18.} {$\quad$ {\bf end if}}
            \State {19.} {$\quad$ {$k=k+1$.}}
		\State {20.} {\bf end while}
		\State {21.} {$\mathbb{S}$=${\rm supp}({\mathbf{A}}^{(N)^{k}})$.}
        \State {22.} {$R=|\mathbb{S}|$.}
		\State {23.} {{\bf for} $i=1:N$}
		\State {24.} {$\quad$ ${\mathbf{A}}^{(i)^{k+1}}={\mathbf{A}}^{(i)^{k+1}}(:,\mathbb{S})$, $\overline{\mathbf{A}}^{(i)^{k}}=\overline{\mathbf{A}}^{(i)^{k}}(:,\mathbb{S})$.}
		\State {25.} {{\bf end for}}
            \State {26.} {\bf for $k=k,k+1,\ldots$}
            \State {27.} {$\quad$ Repeat steps $2-13$ to update $\mathbf{A}^{k+1}$.}
            \State {28.} {\bf end for}
            \State {29.} {\bf return} {$\mathbf{A}^*:={\mathbf{A}}^{k+1}$.}
	\end{algorithmic}
\end{algorithm}

\begin{remark}
    After incorporating the rank reduction strategy in steps $21-25$ of Algorithm \ref{algorithm APGBCDRR}, the subproblem of $\mathbf{A}^{(N)}$ is  transformed into the following unconstrained problem:
\begin{equation*}
		\min_{\mathbf{A}^{(N)}} \frac{1}{2} \|\mathbf{X}_{(N)}-\mathbf{A}^{(N)}\mathbf{D}^{(N)^\top} \|^2,
\end{equation*}
where $\mathbf{D}^{(N)}=(\mathbf{d}_1,\cdots,\mathbf{d}_R)=\mathbf{A}^{(N-1)}\odot\cdots\odot\mathbf{A}^{(1)}$, $\mathbf{X}_{(N)}$ represents the mode-$N$ unfolding of the given tensor $\mathcal{X}$.
\end{remark}
\begin{remark}
     \par (i) We demonstrate that the number of nonzero columns of $\mathbf{A}^{(N)^k}$ remains unchanged after a certain number of iterations through numerical experiments, as shown in Figure \ref{fig:canshu1}.
    \par (ii) The comparison of computation times between Algorithm \ref{algorithm APGBCD1} and Algorithm \ref{algorithm APGBCDRR} in Subsection \ref{chemical data} demonstrates the efficiency of the rank reduction strategy.
\end{remark}
% }

\section{Numerical experiments} \label{sec5}
\indent
\par
In this section, we apply the proposed model and algorithms to the component separation problem in chemometrics. For comparison, we use three other methods in the experiments: AIBCD \cite{WangCOAP},  CP\_ALS \cite{Kolda}, and ATLD \cite{ATLD}.
AIBCD is an accelerated inexact block-coordinate descent algorithm for the nonnegative CP decomposition model with unit length constraints. 
CP\_ALS, also known as the parallel factor analysis (PARAFAC) in chemometrics, is a classical alternating least square CP algorithm for CP decomposition. 
ATLD is an alternating trilinear decomposition method, widely used in chemometrics.
We denote the proposed double-loop block-coordinate proximal gradient descent algorithm with extrapolation for \eqref{modelF} as eDLBCPGD and the algorithm with the rank reduction strategy for \eqref{modelF} as eDLBCPGD\_RR.

All parameters for each method are optimized to ensure fair comparisons and peak performance.
The factor matrices for all algorithms are initialized using random numbers generated by the command $randn(I_n, R)$. 
All experiments are performed on an Intel i9-12300U CPU desktop computer with 16GB of RAM and MATLAB R2023b.
Each algorithm is executed 30 times to evaluate average performance.

For a given tensor $\mathcal{X}\in \mathds{R}^{n_1\times n_2\times n_3}$ and the sequence $\{\mathbf{A}^{(1)^k}, \mathbf{A}^{(2)^k}, \mathbf{A}^{(3)^k}\}$ generated by these algorithms, the relative error at the $k$-th outer-iteration is defined as follows: 
$$
\operatorname{RelErr}_k:=\frac{ \|\mathcal{X}- [\![\mathbf{A}^{(1)^k},{\mathbf{A}^{(2)^k}},\mathbf{A}^{(3)^k}]\!]  \|}{\|\mathcal{X}\|}.
$$
The stopping  criteria are based on the change in relative error:
\begin{equation*} 
 |\operatorname{RelErr}_{k-1}-\operatorname{RelErr}_k |<1 \times 10^{-6}.
\end{equation*}

To evaluate the performance of different methods, the root mean squared error of prediction (RMSEP) \cite{rmsep} is used to measure the results of component separation.
The true concentration profiles matrix is denoted as $\mathbf{A}^{(3)}_{{real}}\in\mathds{R}^{n_3\times R}$ and the computed concentration profiles matrix after regression is represented by $\widetilde{\mathbf{A}}^{(3)}$. The definition of RMSEP  is as follows:
$$
\begin{array}{ll}
RMSEP= \sqrt{\frac{\sum\limits_{i=1}^{n_3}\sum\limits_{j=1}^R(\mathbf{a}^{(3)}_{{\text{real}_{i,j}}}-\tilde{\mathbf{a}}^{(3)}_{_{i,j}})^2}{n_3R}} .
\end{array}$$
\begin{remark}
    If the components cannot be separated correctly, we denote the value of RMSEP  as "-".
  \end{remark}

\subsection{Parameters setting}
\indent
\par
In this subsection, we discuss the parameters of the proposed model and algorithms.

We choose $\lambda$, the parameter of the group sparsity term in the model \eqref{model}, as a decreasing sequence: $\lambda=\max\{1\times 10^{-4}, \kappa \lambda_{max}\}$, where $\kappa=0.97$ and $\lambda_{max}=1000$. 

For the extrapolation parameter $w_k$ in the Algorithm \ref{algorithm APGBCD1} and Algorithm \ref{algorithm APGBCDRR}, we update it as follows: $w_k=\min\{\frac{t_{k-1}-1}{t_k},\gamma\}$, where $t_0=t_{-1}=1$, $t_{k+1}=\frac{1}{2}(1+\sqrt{1+4t_k^2})$, and $\gamma=0.9$. 
Additionally, we choose parameters $\varepsilon=1\times 10^{-5}$ in Algorithm \ref{algorithm for GR1} and $L=20$ in Algorithm \ref{algorithm APGBCDRR}.

For the number of inner-iterations  $m_i, i=1,2,3$ in Algorithm \ref{algorithm for GR1}, we set a uniform value, denoted as $m$, for the subproblem of each factor matrix. 
To investigate the impact of this parameter on computational efficiency, we utilize the two-component data introduced in Subsection \ref{chemical data}.
The range for $m$ is from $1$ to $50$, and the initial estimated number of components is fixed at $5$, i.e., $R=5$.  
We choose the same starting point, and the computation time and number of outer-iterations are reported in Figure \ref{mm}. 
Observations indicate that the computation time decreases significantly when 
$m$ falls within the range of  $5$ to $10$.
Additionally, we observe that the number of outer-iterations decreases as the number of inner-iterations increases and eventually stabilizes. 
Based on these findings, we fix $m=7$ for all subsequent experiments to balance simplicity and efficiency.
 \begin{figure}[htpb]
   \centering
   \subfloat[]{\includegraphics[width=5cm,height=4.2cm]{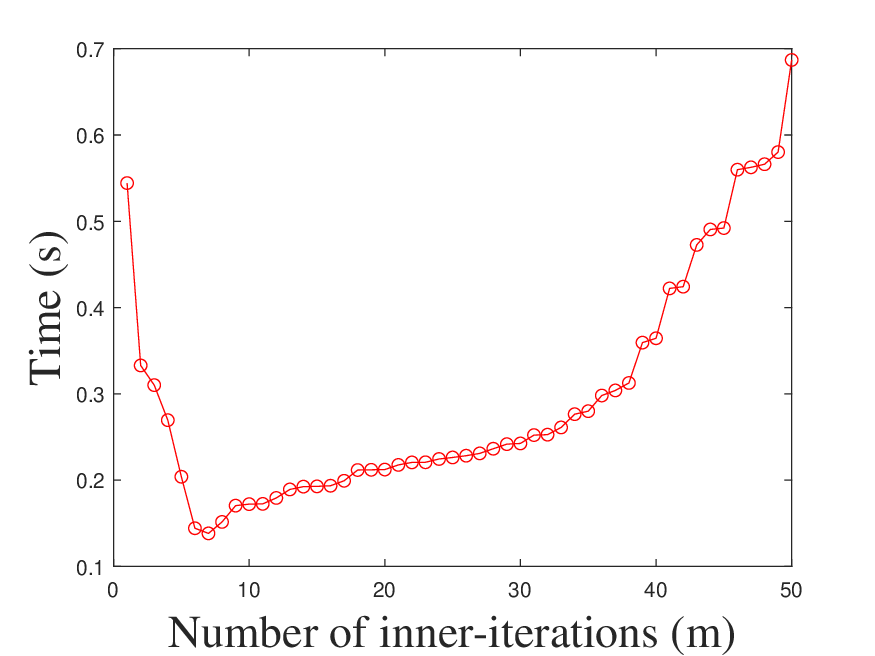}}\hspace{0.2cm}
   \subfloat[]{\includegraphics[width=5cm,height=4.2cm]{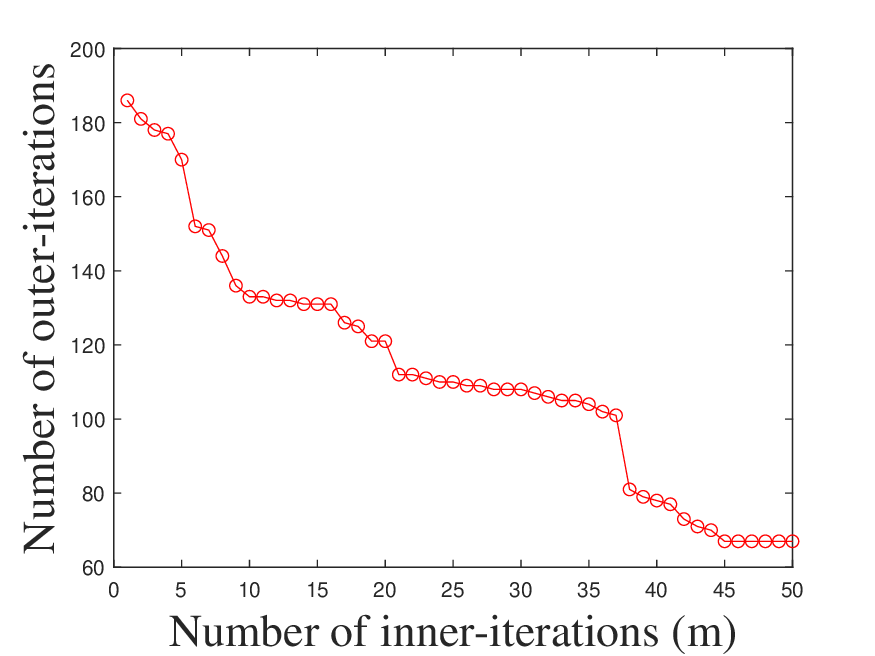}}\hspace{0.2cm}
   \caption{The computation time and number of outer-iterations of Algorithm \ref{algorithm APGBCDRR} for the two-component data with different inner-iterations ($m$).}
   \label{mm}
 \end{figure}

\subsection{Experiments on real data}\label{chemical data}
\indent
\par
In this subsection, we utilize two groups of real chemical data, including two-component data ($32 \times 39 \times 15$) in \cite{Wu} and macrocephalae rhizoma data ($52 \times 214\times \ 16$)\footnote{\url{https://github.com/V-Geler/Conv2dPA/tree/main/Simulator}}, to evaluate the performance of various algorithms.

\subsubsection{Two-component data}
\indent
\par
This data contains two components and exhibits background drift, so it can be treated as containing three components.
We select the initial number of components as $R=  \{3,5,7\}$. 
The RMSEP values and computation time are shown in Table \ref{tab:l2}. 
Figure \ref{fig:two} illustrates the real concentration profiles and the calculated relative concentration profiles resolved by different methods with $R=5$, corresponding to  $\mathbf{A}^{(3)}$. 
Additionally, the change in the number of nonzero columns of $\mathbf{A}^{(3)}$ over the out-iterations is presented in Figure \ref{fig:canshu1}.

\begin{table}[]
 \footnotesize
	\caption{\label{tab:l2} Comparison of computation time (s) and root mean square error of prediction (RMSEP) of different methods for two-component data with $R=3,5,7$. The best results are highlighted
in bold. ``-'' represents that the components are not separated correctly.}
    \begin{center}
	\begin{tabular}{|l|l|l|l|l|l|l|}
	\hline
 \multirow{1}{*}{} & \multirow{1}{*}{Algorithm}     & {AIBCD}    &{CP\_ALS}  & ATLD  & eDLBCPGD   & eDLBCPGD\_RR    \\ \hline
	\multirow{2}{*}{$R=3$}     & Time                & \multicolumn{1}{l|}{{0.032 }}   &
 0.046 & {\bf 0.021} & 0.225 &0.085\\ \cline{2-7} 
 & RMSEP                      &  {\bf 3.607}    &{3.651}   &6.154 &3.935   &3.616 \\ \hline
 
 \multirow{2}{*}{$R=5$}     & Time                & 0.073   &
0.081 &  {\bf 0.063} & 0.239 & 0.141\\ \cline{2-7}  
 & RMSEP                      &  -  &-   &5.913  &   3.697 &{\bf 3.681 } \\ \hline

  \multirow{2}{*}{$R=7$}     & Time                & 0.089  &
0.096 &  {\bf 0.095} & 0.361 &  0.127\\ \cline{2-7}  
 & RMSEP                      &  -  &-   &6.326 & 3.692     &{\bf 3.678}\\ \hline
	\end{tabular}      
\end{center}
	\end{table}

\begin{figure}[htpb]
  \centering
  \subfloat[Real]{\includegraphics[width=4.2cm,height=3.6cm]{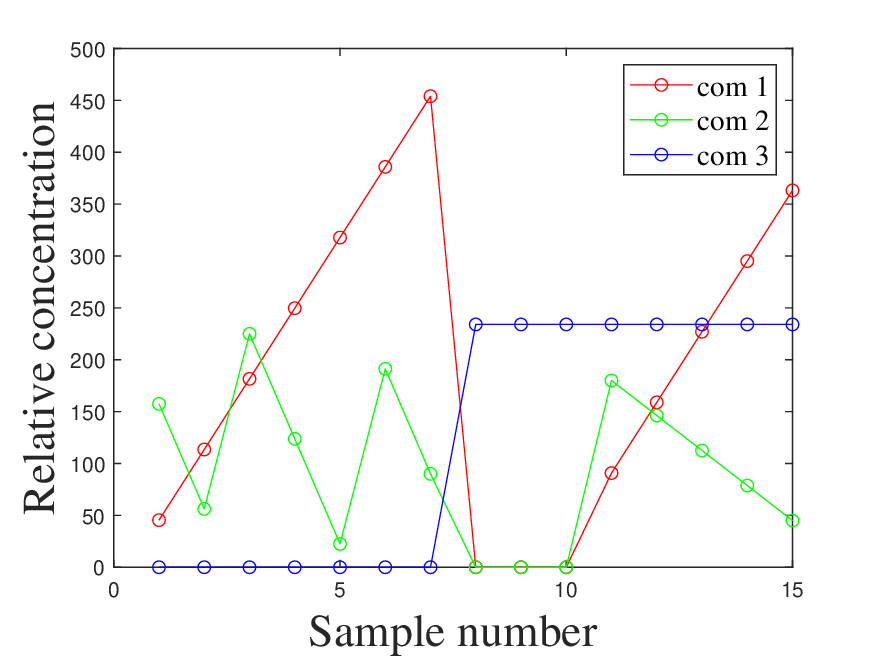}}\hspace{0cm}
  \subfloat[AIBCD]{\includegraphics[width=4.2cm,height=3.6cm]{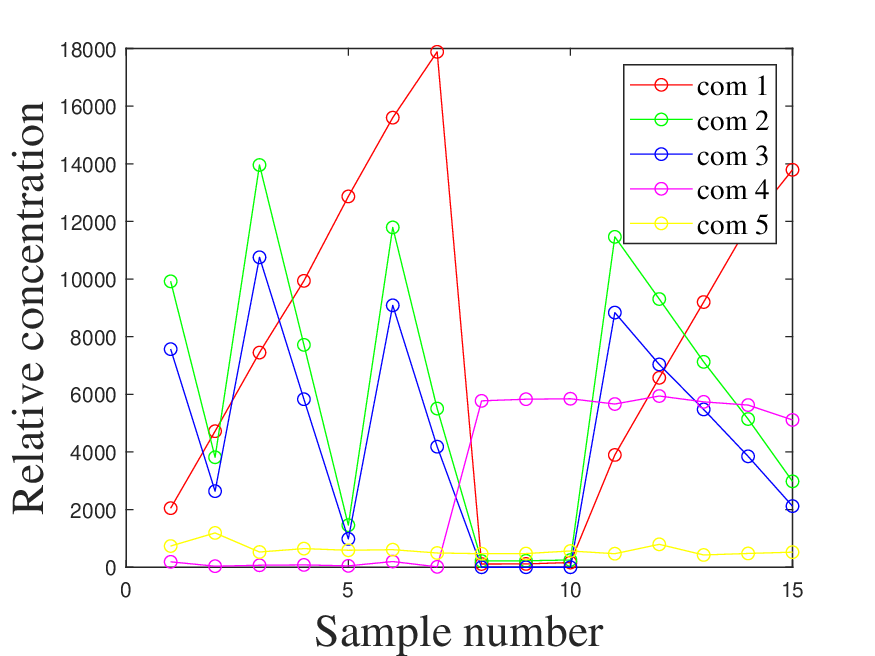}}\hspace{ 0cm}
  \subfloat[CP\_ALS]{\includegraphics[width=4.2cm,height=3.6cm]{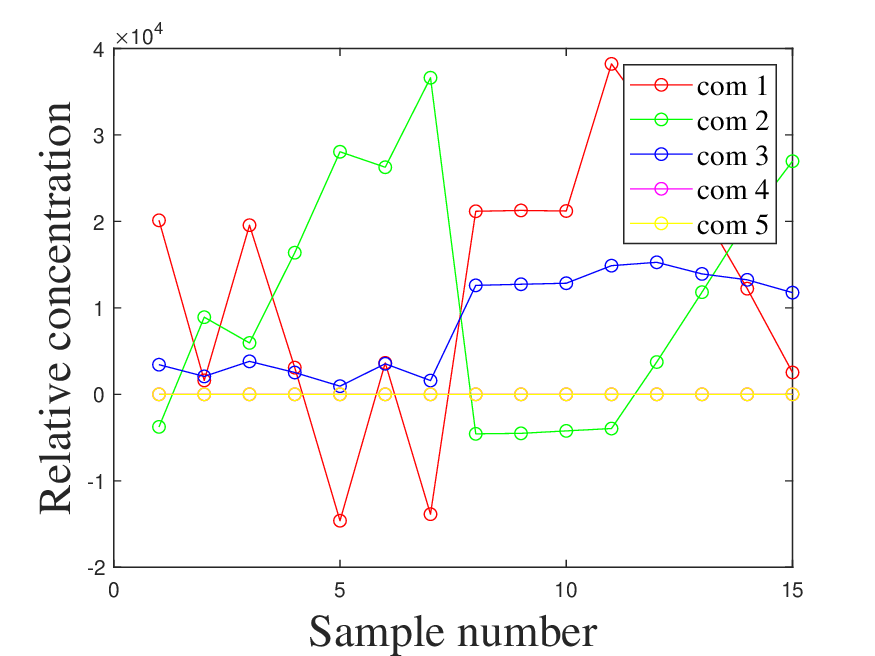}}\hspace{  0cm}
  
  \subfloat[ATLD]{\includegraphics[width=4.2cm,height=3.6cm]{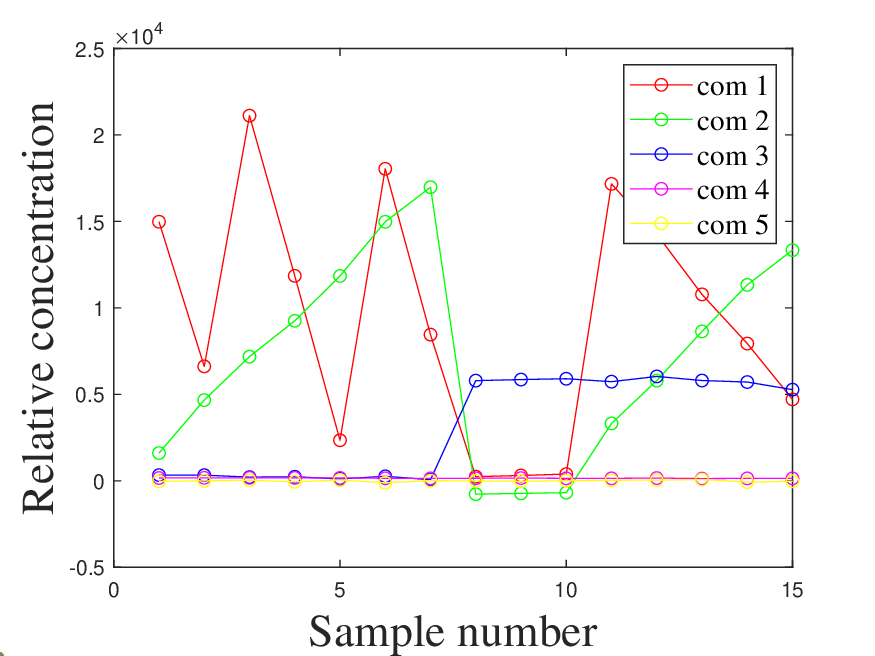}}\hspace{  0.01cm}
  \subfloat[eDLBCPGD]{\includegraphics[width=4.2cm,height=3.6cm]{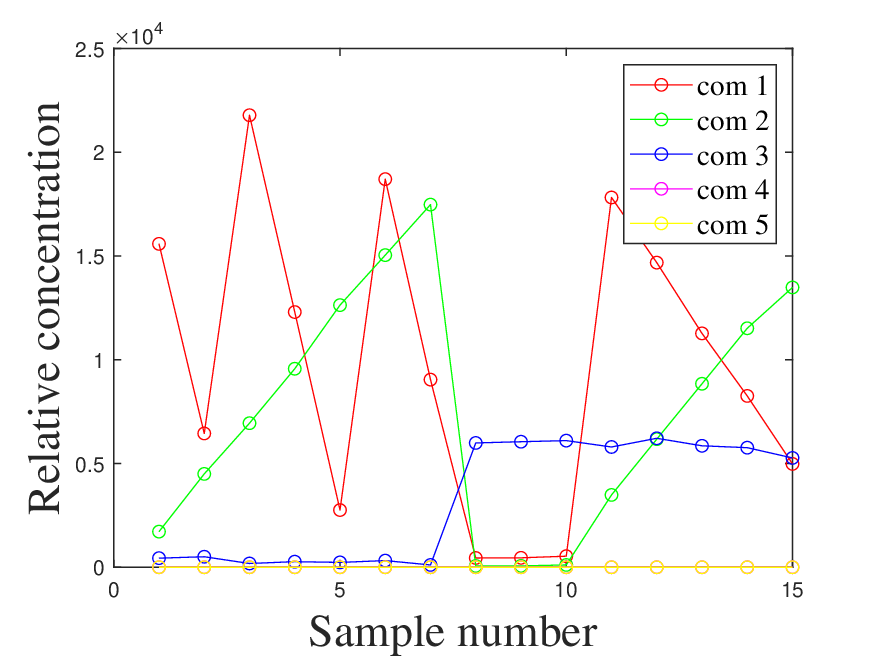}}\hspace{  0.01cm}
  \subfloat[eDLBCPGD\_RR]{\includegraphics[width=4.2cm,height=3.6cm]{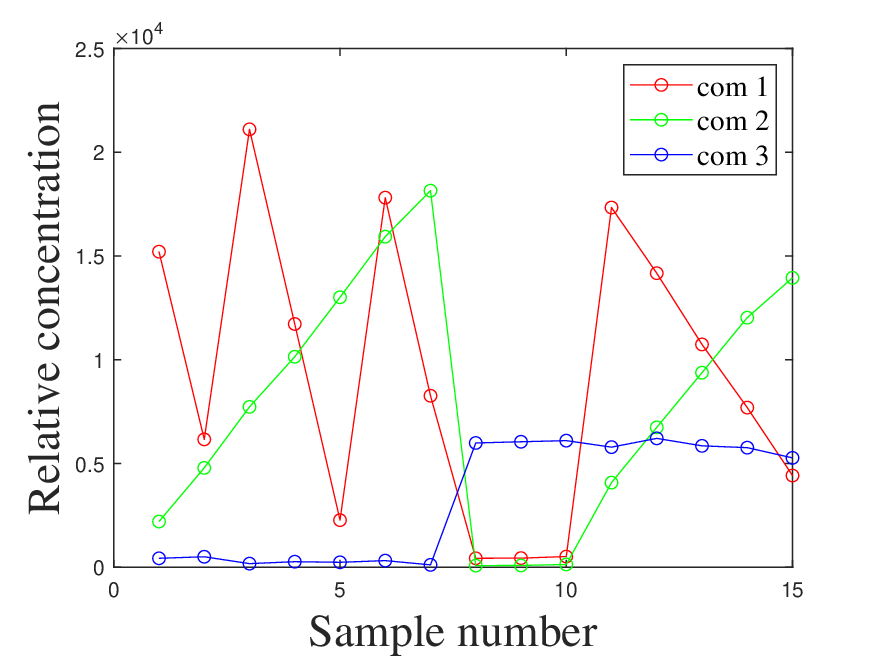}}\hspace{  0cm}

  \caption{Analytical results with $R=5$ for the two-component data.  Real concentration profiles (a) and  relative concentration profiles resolved by (b) AIBCD,
   (c) CP\_ALS, (d) ATLD,  (e) eDLBCPGD, and (f) eDLBCPGD\_RR, respectively.}\label{fig:two}
\end{figure}

\begin{figure}[htpb]
	\centering
	\subfloat[$R=5$]{\includegraphics[width=5cm,height=4.2cm]{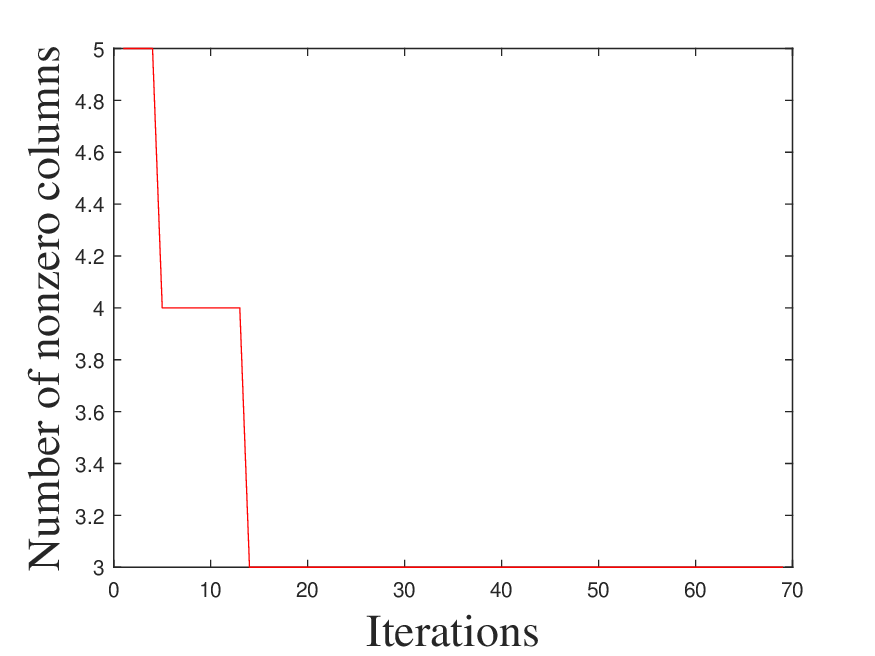}}\hspace{0.02cm}
	\subfloat[$R=7$]{\includegraphics[width=5cm,height=4.2cm]{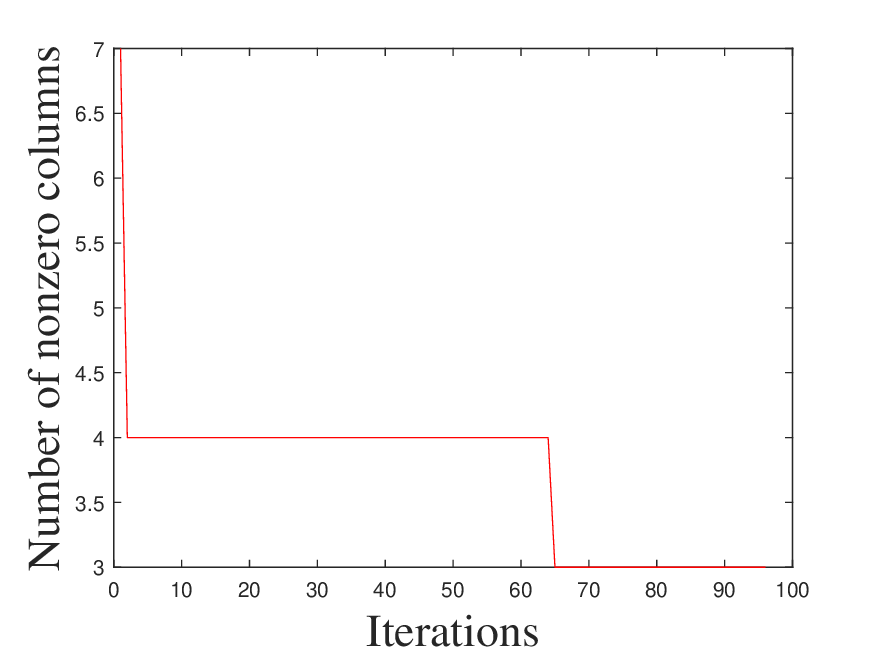}}\hspace{0.02cm}
	\caption{ The number of nonzero columns of the factor matrix $\mathbf{A}^{(3)}$ concerning the iteration numbers.}
	\label{fig:canshu1}
\end{figure}	

Table \ref{tab:l2} shows that AIBCD and CP\_ALS fail to achieve accurate separation when the number of components is overestimated, whereas our proposed methods still exhibit satisfactory results.
Although ATLD provides the fastest computational time and obtains correct separation, it yields unsatisfactory RMSEP values and lacks robust theoretical guarantees. 
Further comparison of eDLBCPGD\_RR with eDLBCPGD shows that eDLBCPGD\_RR requires less computation time, highlighting the effectiveness of the rank reduction strategy. 
Figure \ref{fig:two} shows that when  $R=5$, only our proposed methods consistently separate exactly three components, underscoring the advantage of incorporating the group sparsity term. 
Moreover, Figure \ref{fig:canshu1} reports that the number of nonzero columns of $\mathbf{A}^{(3)}$ decreases and eventually stabilizes over the out-iterations. 
This observation further supports the viability of the rank reduction strategy.	

%%%%%%%%%%%%%%%%%%%-----------------------------------------------------------------------------------------------------------------
\subsubsection{ Macrocephalae rhizoma data}
\indent
\par
The macrocephalae rhizoma data contains three components, and we select the initial number of components as $R=  \{3,5,7 \} $.
The  RMSEP values and computation time are reported in Table \ref{tab:l3}. 
Figure \ref{fig:Baizhu} displays the real concentration profiles and the relative profiles resolved by different methods with $R=7$, corresponding to $\mathbf{A}^{(3)}$. 
Figure \ref{fig:BaizhuA} presents the real chromatographic profiles and the normalized chromatographic profiles resolved by different methods with $R=7$, corresponding to $\mathbf{A}^{(1)}$.
Figure \ref{fig:Baizhub} shows the real spectra profiles and the normalized spectra profiles resolved by different methods with $R=7$, corresponding to  $\mathbf{A}^{(2)}$. 
In Figure \ref{fig:BaizhuA} and Figure \ref{fig:Baizhub}, we only show the chromatographic and spectra profiles of the nonzero components resolved by eDLBCPGD.

\begin{table}[]
			\footnotesize
		\caption{\label{tab:l3}  Comparison of computation time (s) and root mean square error of prediction (RMSEP) of different methods for the Macrocephalae rhizoma data with $R=3,5,7$. The best results are highlighted
in bold. ``-'' represents that the components are not separated correctly.}
    \begin{center}
	\begin{tabular}{|l|l|l|l|l|l|l|}
	\hline
 \multirow{1}{*}{} & \multirow{1}{*}{Algorithm}     & {AIBCD}    &{CP\_ALS}  & ATLD  & eDLBCPGD   & eDLBCPGD\_RR    \\ \hline
	\multirow{2}{*}{$R=3$}     & Time                & \multicolumn{1}{l|}{{\bf 0.022 }}   &
 0.023 &  0.113 & 0.075 &0.041\\ \cline{2-7} 
 & RMSEP                      &  {1.910e-4 }     &{6.320e-4}   &{1.950e-4} &{\bf 1.610e-4 }   &1.730e-4  \\ \hline
 
 \multirow{2}{*}{$R=7$}     & Time                & 0.037   &
 0.043&  0.092 & 0.112 &{\bf 0.053}\\ \cline{2-7}  
 & RMSEP                      &  -  &-   &- & {1.590e-4  }      &{\bf 1.320e-4} \\ \hline

  \multirow{2}{*}{$R=5$}     & Time                & 0.067   &
0.113 &  0.163 & 0.141 &{\bf 0.093}\\ \cline{2-7}  
 & RMSEP                      &  -  &-   &- & 2.010e-4    &{\bf 1.060e-4}\\ \hline
	\end{tabular}      
\end{center}
	\end{table}

\begin{figure}[htpb]
	\centering
	\subfloat[Real]{\includegraphics[width=4.2cm,height=3.6cm]{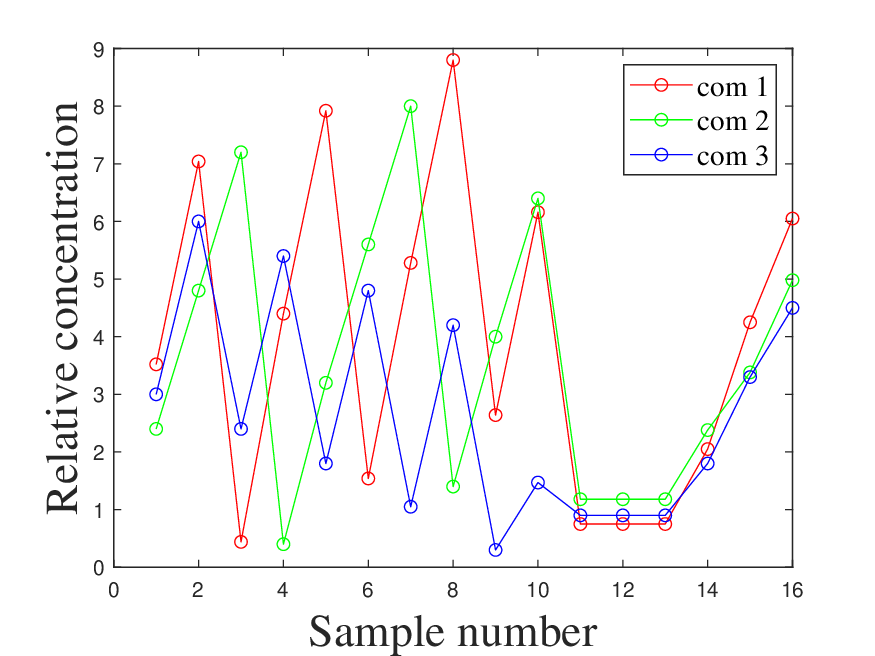}}\hspace{0cm}
	\subfloat[AIBCD]{\includegraphics[width=4.2cm,height=3.6cm]{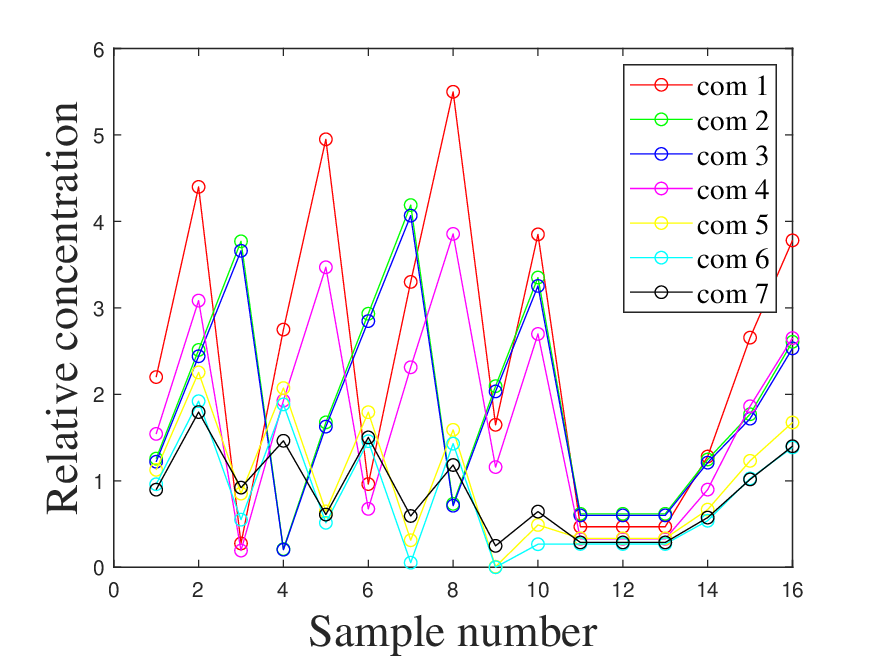}}\hspace{0cm}
	\subfloat[CP\_ALS]{\includegraphics[width=4.2cm,height=3.6cm]{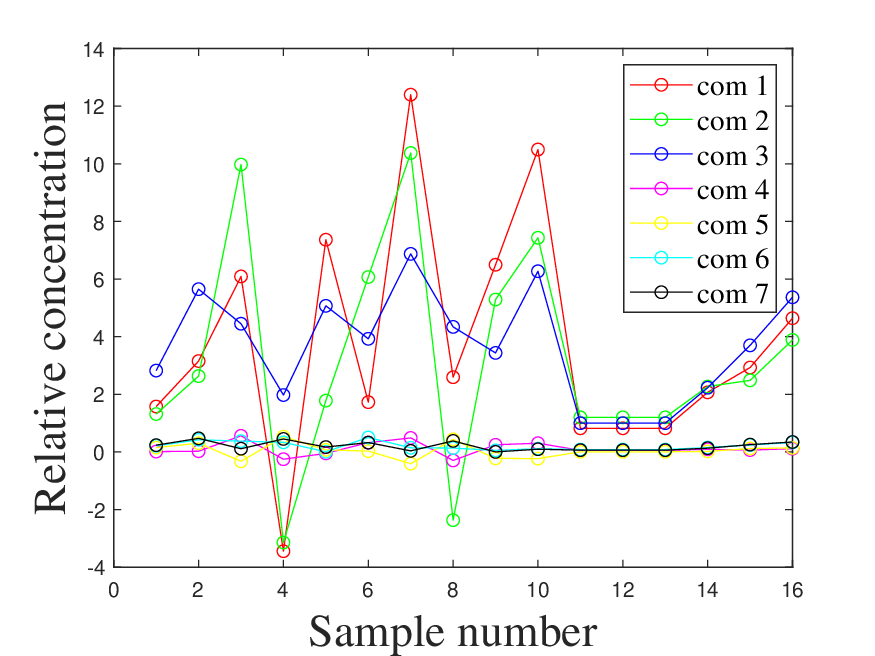}}\hspace{0cm}
    
	\subfloat[ATLD]{\includegraphics[width=4.2cm,height=3.6cm]{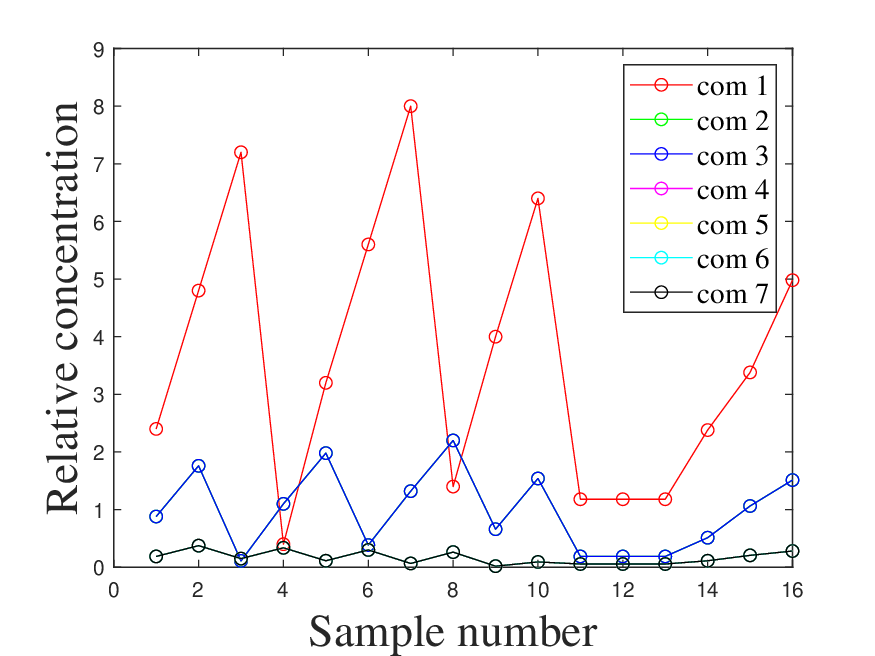}}\hspace{  0cm}
	\subfloat[eDLBCPGD]{\includegraphics[width=4.2cm,height=3.6cm]{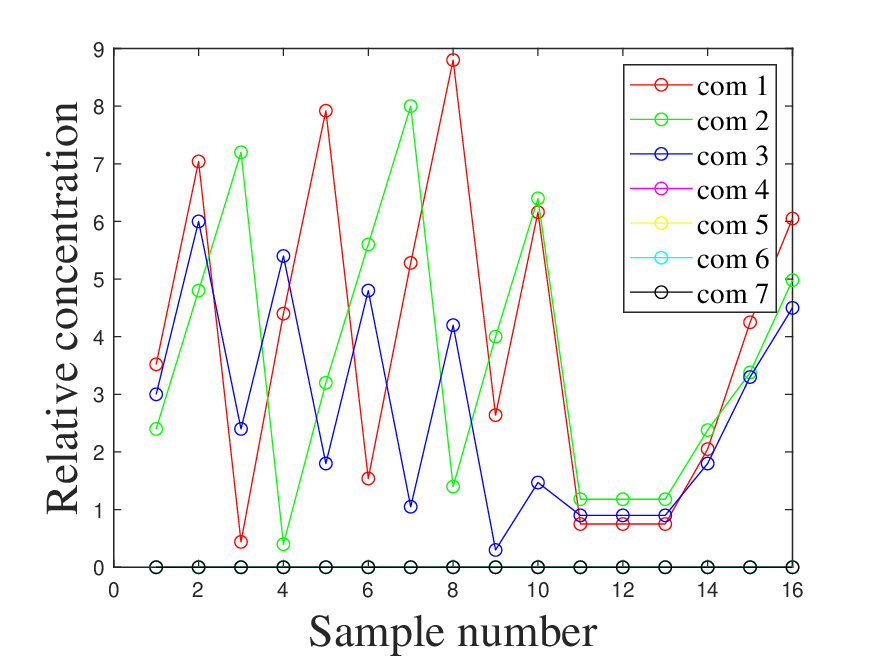}}\hspace{  0cm}
	\subfloat[eDLBCPGD\_RR]{\includegraphics[width=4.2cm,height=3.6cm]{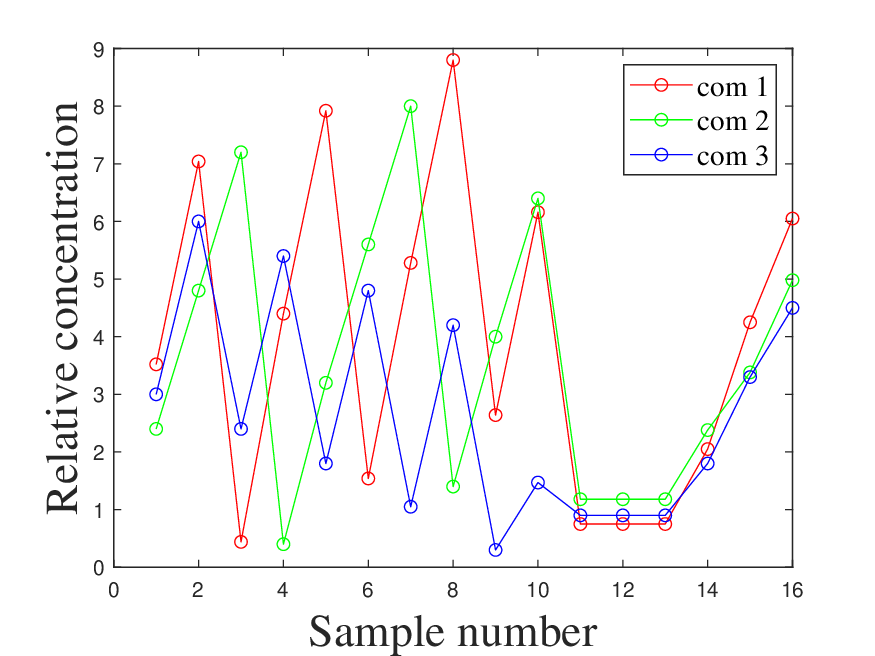}}\hspace{ 0cm}
  
	\caption{Analytical results with $R=7$ for the macrocephalae rhizoma data.  Real concentration profiles (a) and  relative concentration profiles 
	resolved by (b) AIBCD,
	(c) CP\_ALS, (d) ATLD,  (e) eDLBCPGD, and (f) eDLBCPGD\_RR, respectively.
	}\label{fig:Baizhu}
  \end{figure}

 \begin{figure}[ht]
	\centering
	\subfloat[Real]{\includegraphics[width=4.2cm,height=3.6cm]{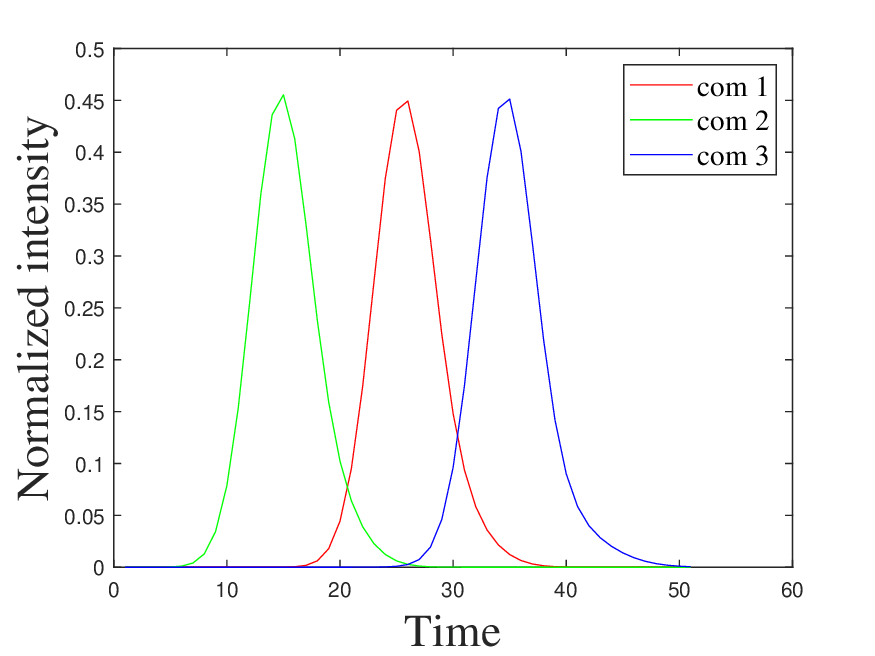}}\hspace{  0cm}
	\subfloat[AIBCD]{\includegraphics[width=4.2cm,height=3.6cm]{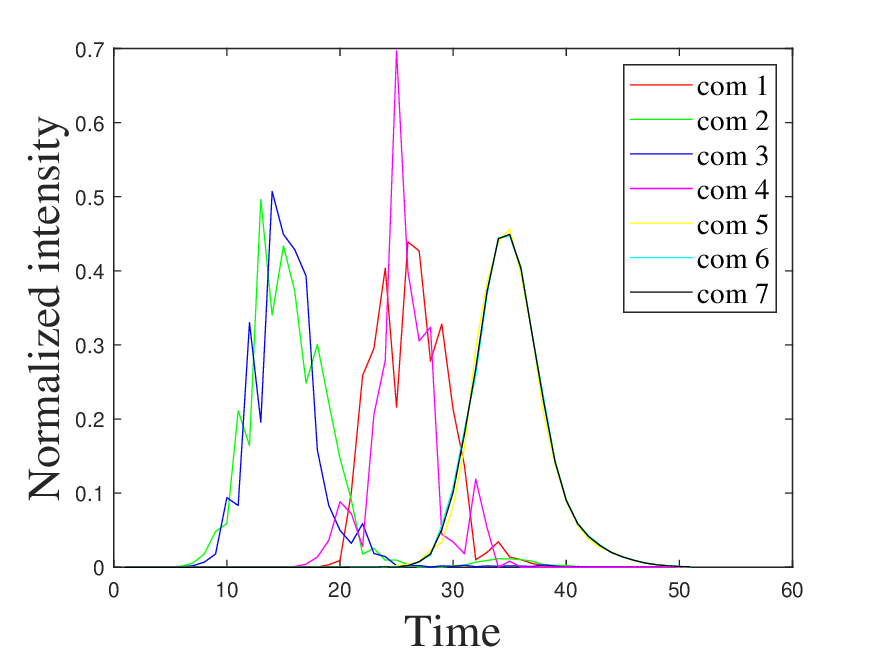}}\hspace{  0cm}
	\subfloat[CP\_ALS]{\includegraphics[width=4.2cm,height=3.6cm]{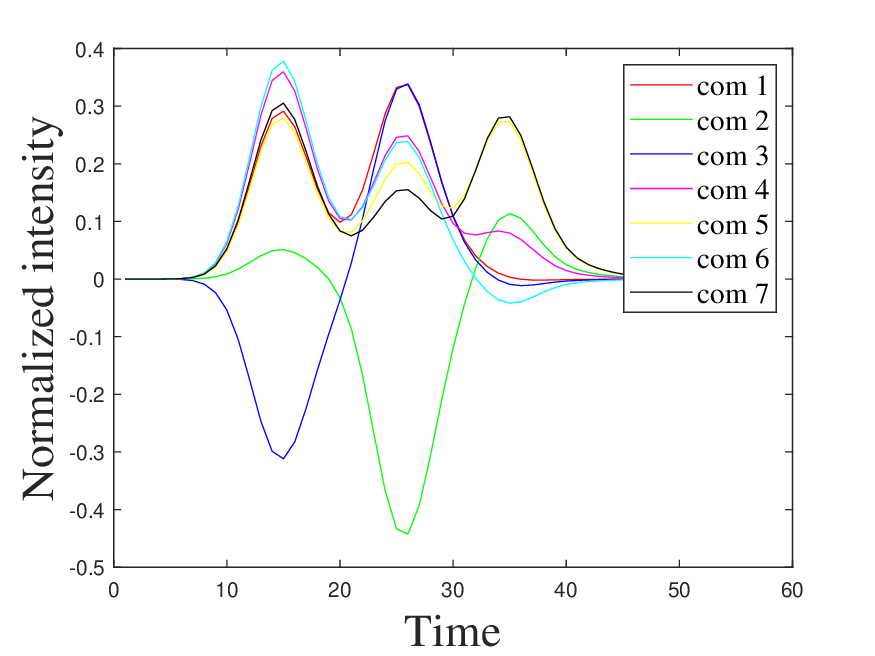}}\hspace{  0cm}
    
	\subfloat[ATLD]{\includegraphics[width=4.2cm,height=3.6cm]{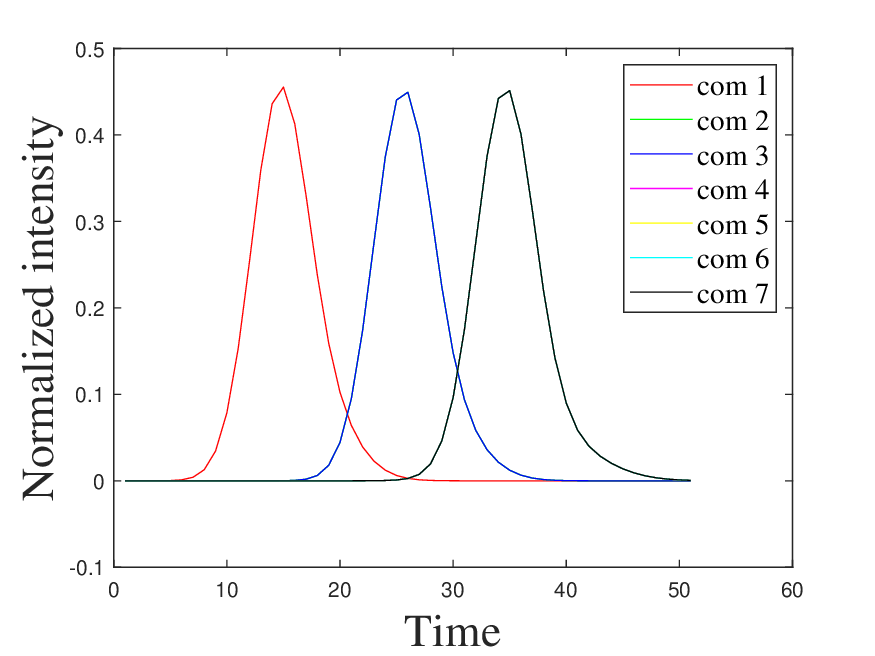}}\hspace{  0cm}
	\subfloat[eDLBCPGD]{\includegraphics[width=4.2cm,height=3.6cm]{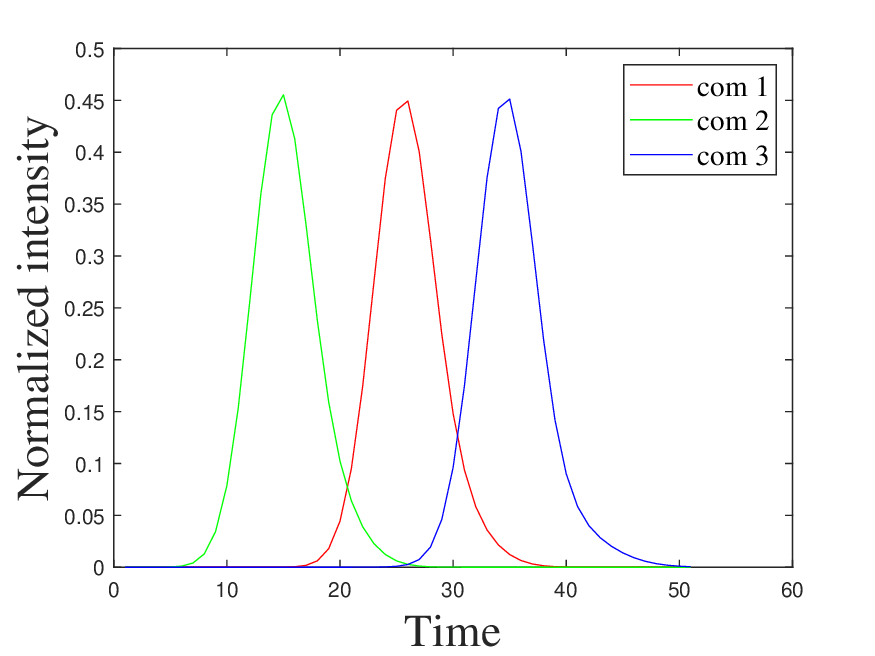}}\hspace{  0cm}
	\subfloat[eDLBCPGD\_RR]{\includegraphics[width=4.2cm,height=3.6cm]{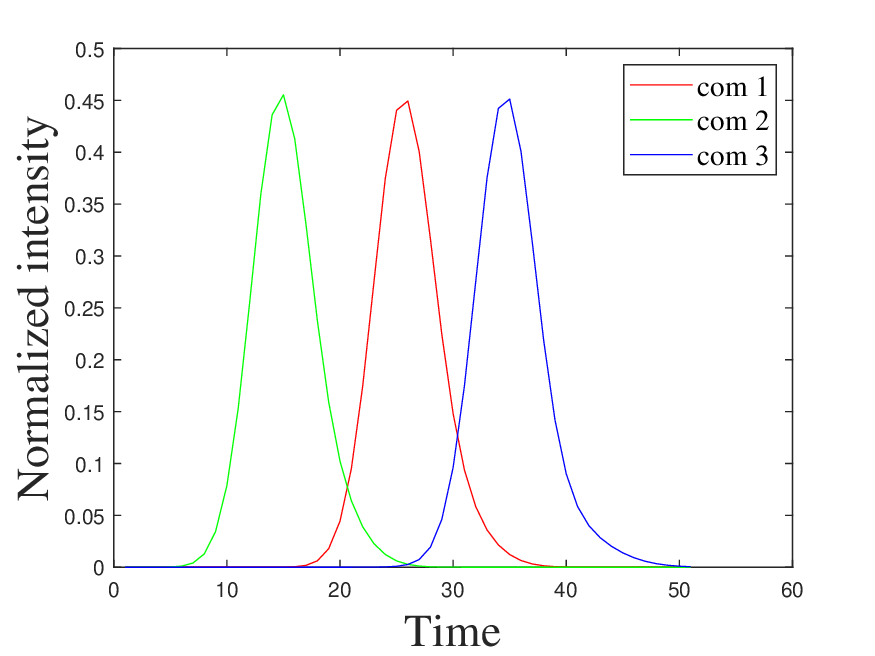}}\hspace{ 0cm}
 
	\caption{Analytical results with $R=7$ for the macrocephalae rhizoma data.  
	Real chromatographic profiles (a) and  normalized chromatographic profiles
	resolved by (b) AIBCD,
	(c) CP\_ALS, (d) ATLD,  (e) eDLBCPGD, and (f) eDLBCPGD\_RR, respectively.}\label{fig:BaizhuA}
 \end{figure}

\begin{figure}[ht]
	\centering
	\subfloat[Real]{\includegraphics[width=4.2cm,height=3.6cm]{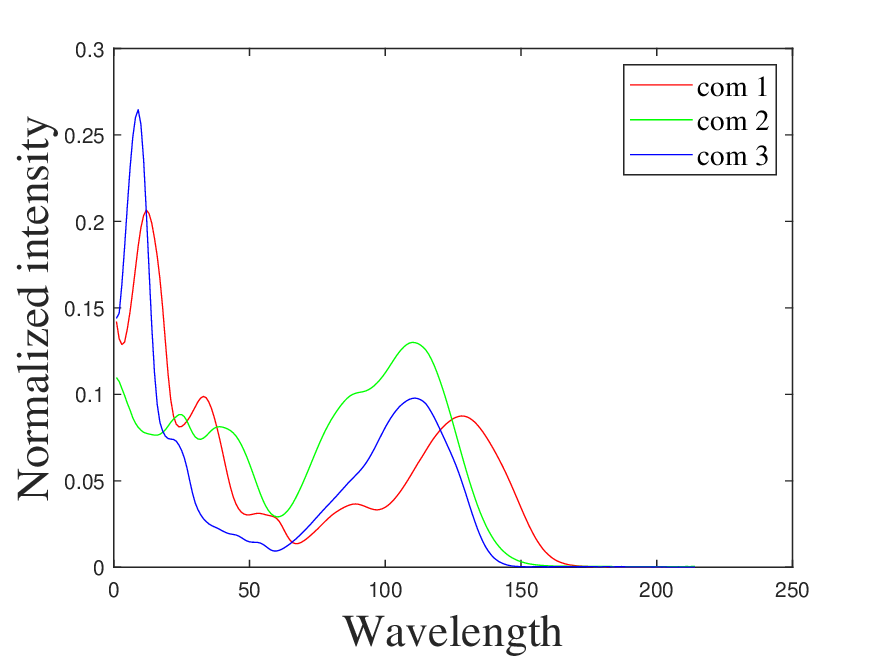}}\hspace{  0cm}
	\subfloat[AIBCD]{\includegraphics[width=4.2cm,height=3.6cm]{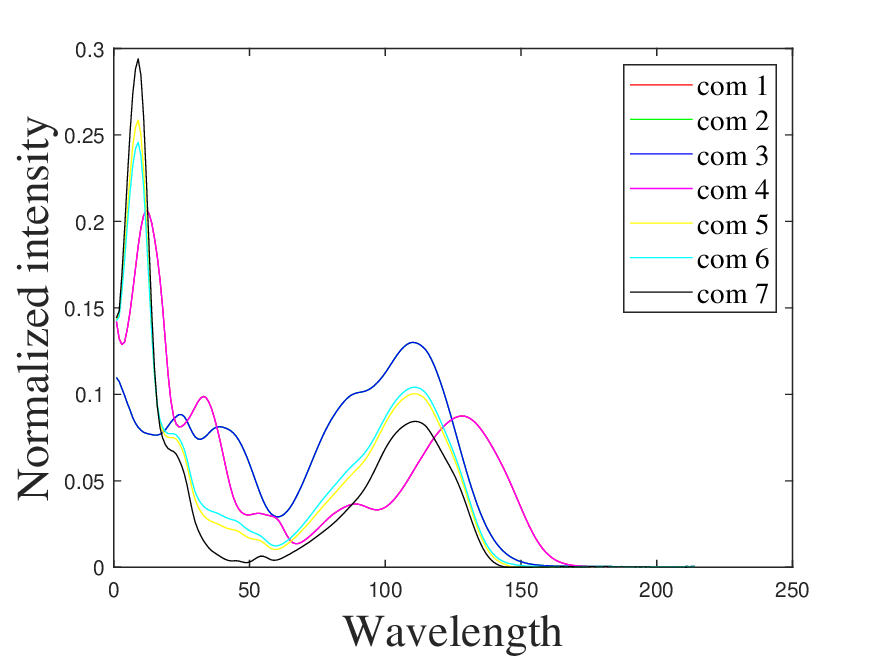}}\hspace{  0cm}
	\subfloat[CP\_ALS]{\includegraphics[width=4.2cm,height=3.6cm]{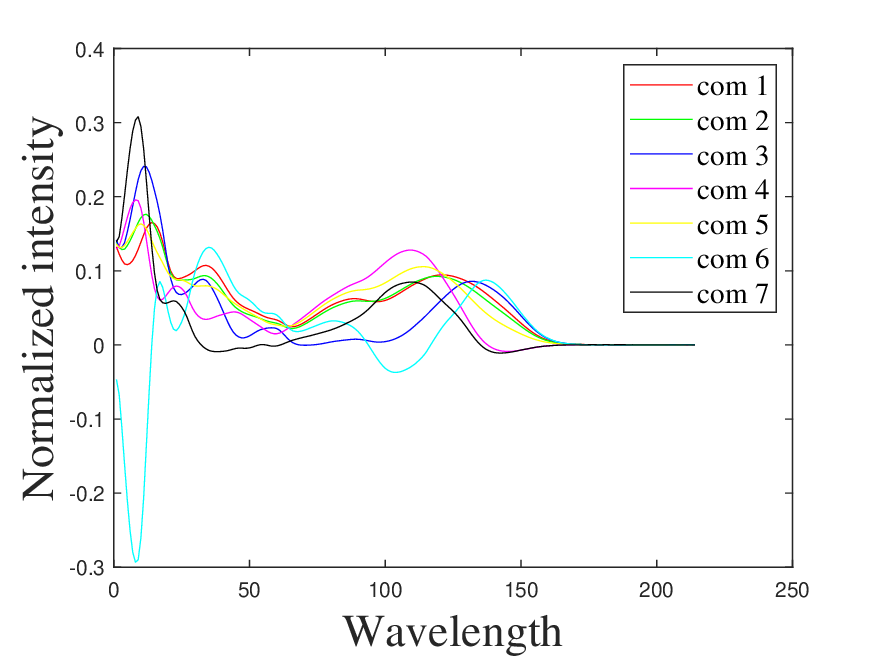}}\hspace{  0cm}
    
	\subfloat[ATLD]{\includegraphics[width=4.2cm,height=3.6cm]{FIGC1/BN7.eps}}\hspace{  0cm}
	\subfloat[eDLBCPGD]{\includegraphics[width=4.2cm,height=3.6cm]{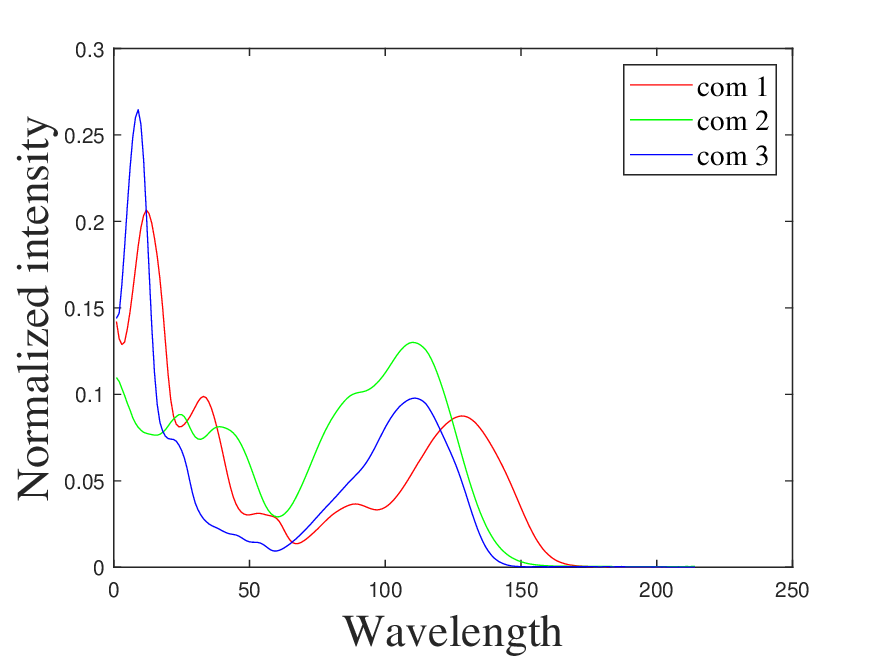}}\hspace{  0cm}
	\subfloat[eDLBCPGD\_RR]{\includegraphics[width=4.2cm,height=3.6cm]{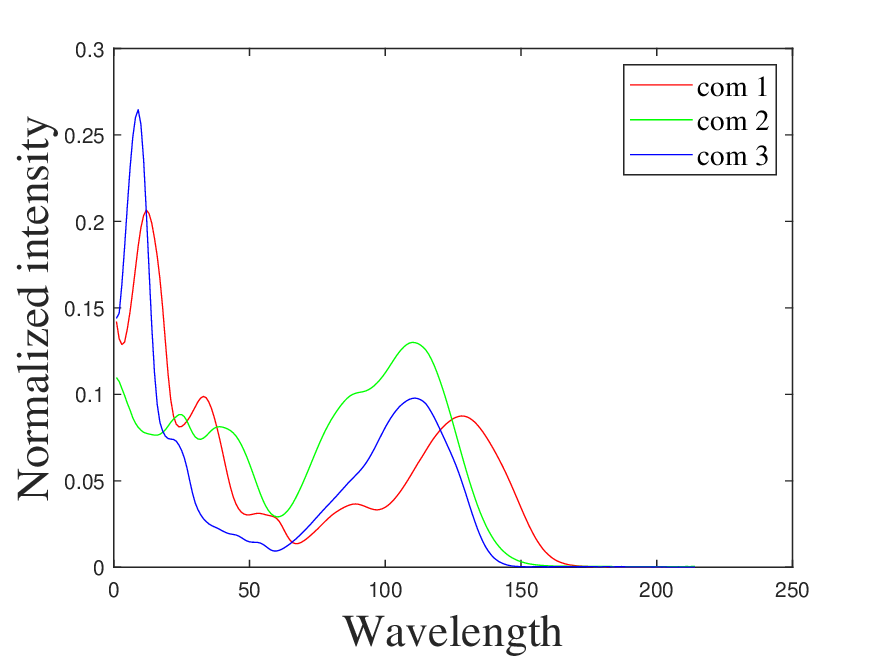}}\hspace{ 0cm}
 
	\caption{Analytical results with $R=7$ for the macrocephalae rhizoma data.  
	Real spectra profiles (a) and  normalized spectra profiles
	resolved by (b) AIBCD,
	(c) CP\_ALS, (d) ATLD,  (e) eDLBCPGD, and (f) eDLBCPGD\_RR, respectively.
		}\label{fig:Baizhub}
 \end{figure}

As shown in Table \ref{tab:l3}, only our proposed methods can achieve satisfactory RMSEP  when the initial number of components is set to $5$ and $7$. 
Furthermore, the performance of our proposed methods is comparable to other methods when the initial number of components is accurate. 
These findings highlight the effectiveness and robustness of our proposed methods. 
% ATLD exhibits slower computation times for this data due to its calculations involving the Moore-Penrose pseudoinverse of a matrix. 
The eDLBCPGD\_RR algorithm demonstrates reduced computational time compared to eDLBCPGD, validating the efficiency of the rank reduction strategy.
Figure \ref{fig:Baizhu}, Figure \ref{fig:BaizhuA}, and Figure \ref{fig:Baizhub} illustrate that our methods can accurately recover accurate concentration profiles, chromatographic profiles, and spectra profiles with $R=7$, whereas other methods fail to do so. 
In particular, although ATLD can obtain the spectra of the components, it divides the concentration of a component into equal parts.
 
	\section{Conclusions} \label{sec6}
    \indent
\par
	Most existing CP decomposition-based models and algorithms typically require the accurate estimation of the CP rank, which is a challenging task.
    To address this challenge, in this paper, we establish the equivalence between the CP rank and the minimization of the group sparsity on any of the factor matrices under the unit length constraints on the columns of the others.
    Based on this theoretical foundation, we develop a regularized CP decomposition model to obtain a rank CP decomposition of a given tensor.
    To solve this model, we introduce a double-loop block-coordinate proximal gradient descent algorithm with extrapolation and provide the convergence analysis. 
    Furthermore, based on the proof that the nonzero columns of $\mathbf{A}^{(N)}$ remain unchanged after a certain number of iterations, we develop a rank reduction strategy to reduce the computational complexity. 
    Finally, we apply the proposed model and algorithms to the component separation problem in chemometrics.
    Numerical experiments conducted on real data demonstrate that our proposed methods can find the optimal solutions even when the initial number of components is overestimated, highlighting the robustness and effectiveness of our approach.

  \section*{Acknowledgments}
   The authors would like to express their gratitude to Professor Hailong Wu, Tong
	Wang,  and Anqi Chen for generously sharing the code of ATLD and the used chemical data.

\bibliographystyle{plain}
\bibliography{CPGS}
\clearpage

\end{document}